\UseRawInputEncoding
\documentclass[a4paper,11pt]{article}
\usepackage{makeidx}
\usepackage[mathscr]{eucal}
\usepackage[dvips]{color}
\usepackage{amssymb, amsfonts, amsmath, amsthm, graphicx}
\usepackage{here}
\usepackage{tasks}
\usepackage{caption}
\usepackage{ esint }

\newtheorem{proposition}{Proposition}[section]

\newtheorem{theorem}{Theorem}[section]

\newtheorem{lemma}[proposition]{Lemma}
\newtheorem{remark}{Remark}[section]
\newtheorem{corollary}[theorem]{Corollary}

\newcommand{\N}{\mathbb{N}}

\newcommand{\R}{\mathbb{R}}

\newcommand{\boF}{\mathcal{F}}

\newcommand{\boJ}{\mathcal{J}}
\newcommand{\boK}{\mathcal{K}}
\newcommand{\boL}{\mathcal{L}}

\newcommand{\boO}{\mathcal{O}}

\newcommand{\boS}{\mathcal{S}}
\newcommand{\boT}{\mathcal{T}}

\newcommand{\boW}{\mathcal{W}}

\numberwithin{equation}{section}

\title{Infinitely many positive solutions of a Gross-Pitaevskii equation in the presence of a harmonic potential and combined nonlinearities }
\author{Yakine Bahri and Hichem Hajaiej}
\date{}
\begin{document}

\maketitle

\begin{abstract}
The main goal of this paper is to address an important conjecture in the field of differential equations in the presence of a harmonic potential. While in the subcritical case, the uniqueness of positive solution has been addressed by Hirose and Ohta in \cite{HiroseOhta} in 2007, the problem has remained open for years in the supercritical case. In \cite{Foued2}, the author obtained interesting numerical computations suggesting that for some bifurcating parameter $\lambda$, the equation has many positive solutions that vanish at infinity. In this paper, we provide a proof to this claim by constructing an accountable number of solutions that bifurcate from the unique singular solutions with $\lambda$ close to the first eigenvalue $\lambda_1$ of the harmonic operator $-\Delta + |x|^2$. Our method hinges on a matching argument, and applies to the supercritical case, and to the supercritical case in the presence of a subcritical, critical or supercritical perturbation. 
\end{abstract}

%\newpage

\section{Introduction}
We consider the following scalar field equation:

\begin{equation}
\label{SFE-1}
\left\{ \begin{array}{c} \Delta u -( |x|^2-\lambda ) u + |u|^{q-1}u + |u|^{p-1}u=0 ,\\
u >0,\\
\lim_{r\to \infty} u(r)=0,
\end{array}
\right.
\end{equation}
where $d\geq3$, $\frac{d+2}{d-2}<p<p_{\rm JL}:=\left\{\begin{array}{l}+\infty \text { for } d \leq 10 \\ 1+\frac{4}{d-4-2 \sqrt{d-1}} \text { for } d \geq 11\end{array}\right.$ and  $1<q < \frac{p+1}{2} $. $|u|^{q-1}u$ is seen as a perturbation of the original equation: 

\begin{equation}
\label{SFE}
\left\{ \begin{array}{c} \Delta u -( |x|^2-\lambda ) u  + |u|^{p-1}u=0 ,\\
u >0,\\
\lim_{r\to \infty} u(r)=0,
\end{array}
\right.
\end{equation}

This equation models the propagation of a laser beam in an optical fiber. It also appears in several crucial applications like the theory of Bose-Einstein condensates. It is referred to as the Gross-Pitaevskii equation.

We denote by $\lambda_1$ the first eigenvalue of the harmonic operator $-\Delta + |x|^2$ and $\Sigma := \{u \in H^1(\R^d); |x| u \in L^2(\R^d) \}$. By the breakthrough result of Li and Ni \cite{Li-Wei}, we know that any solution of \eqref{SFE-1} is radially symmetric. Therefore \eqref{SFE-1} is equivalent to the following reduced ordinary differential equation:
\begin{equation}
\label{ODE-pq}
\left\lbrace \begin{array}{c}\partial_{rr}u+\frac{d-1}{r} \partial_{r}u -( r^2-\lambda ) u + |u|^{q-1}u + |u|^{p-1}u=0, \quad r>0,\\
u >0,\\
\lim_{r\to \infty} u(r)=0,
\end{array}
\right.
\end{equation}
In the case of a single non-linearity, this equation becomes
\begin{equation}
\label{ODE-p}
\left\lbrace \begin{array}{c}\partial_{rr}u+\frac{d-1}{r} \partial_{r}u -( r^2-\lambda ) u + |u|^{p-1}u=0, \quad r>0,\\
u >0,\\
\lim_{r\to \infty} u(r)=0,
\end{array}
\right.
\end{equation}
It has a no nontrivial solution in $\Sigma$ for $\lambda \geq \lambda_1$. In the subcritical case $p\in (1,2^*-1)$ where $2^*=\frac{2d}{d-2}$, Hirose and Ohta \cite{HiroseOhta} showed the existence and uniqueness of the solutions when $\lambda<\lambda_1=d$. Hadj Selem et al addressed the existence and symmetry of ground state solutions and showed the optimality of their conditions. They also addressed the orbital stability of standing waves \cite{HHMT}.

We denote $R_1:=\{(\lambda,p) \in \R \times \R \ : \ \lambda < d \  \text{and} \  p\in (1,2^*-1) \}$. For $p \geq 2^*-1$, \eqref{ODE-p} has no nontrivial solutions for $\lambda \leq 0$. Therefore in the framework of this study, we will only consider $0<\lambda<d$ and denote $R_2:=\{(\lambda,p) \in \R \times \R \ : \ \lambda < d \  \text{and} \  p\geq 2^*-1 \}$.

In $R_1$, it has been proven by Hirose and Ohta \cite{HiroseOhta} that \eqref{ODE-p} has a unique solution. In $R_2$, F. Hadj Selem \cite{Foued2} obtained some numerical results that suggest the following claim :"there exists some $0<\lambda<d$ such that equation \eqref{ODE-p} does not have a unique positive solution". Despite the importance of such claim in understanding the supercritical case of a class of Gross-Pitaeveskii equations, the literature remained silent for years. In this paper, we provide an answer to this question by constructing an accountable number of solutions to \eqref{ODE-p} that bifurcate from the unique singular solution when $(\lambda,p)\in R_2$. This solution has been determined when $\lambda=\lambda_*\in (0,d)$ in \cite{HSKW}.
%Additionally, the authors were able
%with $\lambda > -d$ and $1<p<p_{JL}$. Here $d$ corresponds to the dimension i.e. $x\in \R^{d}$ with $r=|x|$.
%We take a solution of the form
%$$u=\Phi + \varepsilon \psi + \varepsilon w ,$$

The authors also showed that the behavior of the singular solution $\Phi_p$ near
the origin is given by
\begin{equation*}
\Phi_p(r)=A(p,d)r^{-\frac{2}{p-1}}\left(1 +O(r^2) \right) \quad {\rm as } \ r\to 0,
\end{equation*}
where
\begin{equation}
\label{def-Apd}
A(p,d):=\left(\frac{2}{p-1} \left(d-2-\frac{2}{p-1}\right)\right)^{\frac{1}{p-1}}.
\end{equation}

In our case, we denote by $\Phi$ the singular solution to \eqref{ODE-pq} for $\lambda=\lambda_*$. We use their argument in order to prove its existence and behavior near the origin. This is stated in the following theorem. (See the appendix for the sketch of its proof). Note that, we believe that the uniqueness of $\lambda_*$ does not depend on the perturbation term $|u|^{q-1}u$ in our case. The proof should be similar to the equation in  can be proven exactly as in the \cite{HSKW}. However, in \cite{BFPS}, Pelinovsky et al mentioned that its proof is not completely correct. This is why we decided to not discuss the uniqueness in this paper since it is not crucial for our argument.

\begin{theorem}
\label{Thm-sing}
Let $1<q<p.$ There exists an eigenvalue $\lambda_{*} \epsilon$ $\left(0, \lambda_{1}\right)$ such that \eqref{ODE-pq} has a unique radial singular solution $\Phi$. Its asymptotic behavior near the origin is given by
\begin{equation}
\label{behav:phi:origin}
\Phi(x)=A(p, d)|x|^{-2 /(p-1)}\left\{1+\boO\left(|x|^{\frac{2(p-q)}{p-1}}\right)\right\} \quad \text { as } \quad x \rightarrow 0
\end{equation}
with
$A(p,d)$ is defined in \eqref{def-Apd}. Moreover,
\begin{equation}
\label{behav:deriv-phi:origin}
\Phi'(x)= -\frac{2A(p, d)}{p-1} |x|^{-\frac{2}{p-1}-1}\left\{1+\boO\left(|x|^{\frac{2(p-q)}{p-1}}\right)\right\} \quad \text { as } \quad x \rightarrow 0.
\end{equation}

\end{theorem}

Our aim is to construct an accountable number of solutions to \eqref{ODE-pq} that bifurcate from $\Phi$ with $\lambda$ close to $\lambda_*$.  As a consequence, we conclude that there exists an eigenvalue $\lambda \in (0,d)$ such that the equation \eqref{SFE-1} does not have a unique positive solution that vanishes at infinity. The main result is given by the following theorem.

\begin{theorem}
\label{thm:matching}
Let  $d\geq3$, $\frac{d+2}{d-2}<p<p_{\rm JL}$ and  $1<q< \min(\frac{p+1}{2},p).$ There exists $N\in \N$ large enough such that for all $n\geq N$, there exist and eigenvalue $\lambda_n$ and a smooth solution $u_{\theta_n}$ to \eqref{ODE-pq} with initial condition $\theta_n \to \infty $ and $\lambda_n \to \lambda_*$ when $n\to \infty $ . In addition,
$$\lim_{n\to \infty} u_{\theta_n} = \Phi \quad {\rm in} \ H^1(\R^d). $$
Moreover, for any large $n$, we have
$$\theta_n = \tilde{\theta}_n \left(1+ O\left( \tilde{\theta}_n^{\frac{1-\sigma}{\sigma}(\beta-q)}  \right) \right) $$
and
\begin{equation}
\label{Lambda-theta}
\lambda_n = \lambda_* + C_1 \tilde{\theta}_n^{\frac{1-\sigma}{\alpha}} \sin\left( \alpha \omega \log\left(\theta_n \right)  + C_2\right) \left( 1+ O\left( \tilde{\theta}_n^{\frac{1-\sigma}{\sigma}(\beta-q)}  \right) \right) 
\end{equation}
where $\alpha:=\frac{2}{p-1}$, $ \omega:=\sqrt{|\left(d-2\right)^2 -4 pA^{p-1}}|$, $\sigma=\frac{d}{2}-\frac{2}{p-1}$, $\beta=\frac{p+1}{2}$, $\tilde{\theta}_n= C_3 \exp \left( \frac{p-1}{\omega} n \right) $, $C_1>0$, $C_2$ and $C_3>0$ independent on $n$.
\end{theorem}

\begin{remark}
The condition $p<p_{\rm JL}$ is crucial for oscillating behavior of the graph of the eigenvalue versus the initial data \eqref{Lambda-theta}. This will not be the case when $p>p_{\rm JL}$. More precisely, as in \cite{BFPS}, we believe that when $p\geq p_{JL}$ the behavior of the eigenvalue graph should be monotone. This can be implicitly derived from the proof of Lemma \ref{lem:ker-L}. The oscillatory behavior is due to $\left(d-2\right)^2 -4 pA^{p-1} <0$ when $p<p_{JL}$. This  behavior is a consequence of the behavior near the origin of the functions in the kernel of $\mathcal{L}$ (see \eqref{behv:origin} for more details). Those functions are monotone when $\left(d-2\right)^2 -4 pA^{p-1}\geq 0$. As in \cite{BFPS} the uniqueness of positive radial solution should hold when the graph of the eigenvalue parameter versus the initial data is monotone i.e. when $p\geq p_{JL}$. We do not cover this case in our paper.
\end{remark}

In \cite{BFPS} Pelinovsky et al showed this result without the perturbation term $|u|^{q-1}u$. We believe that their approach is different and does not apply to our case \cite{ABSA, MAL}.  Our interest in considering \eqref{SFE-1} is motivated by many facts. First, mixed power nonlinearities arise naturally in the theory of Bose-Einstein condensates and nonlinear optics. Secondly, as \cite{BFPS} addressed the pure nonlinearity case, we wanted to investigate the validity of their results if we perturb the equation with a subcritical, critical, or supercritical nonlinearity.  Our method is self-contained and seems simpler and covers their case. In addition, in \cite{BFPS} Pelinovsky et al uses two assumptions, which were checked numerically, in order to guarantee that each $\lambda_*$ is isolated. Our argument by pass this technical assumption and our proof works also for a continuous set of $\lambda_*$. Lastly, combined ( or mixed) power nonlinearities is a fast-growing field due to its numerous applications, \cite{J, L, S1, S2, S}, to only list a few papers in the last months. In the presence of a Harmonic potential, Kikuchi \cite{Kikuchi} considered sub-critical energy nonlinearities with opposite signs. He proved the existence of standing waves. However, so far, there are no results for super-critical nonlinearities. Therefore, we wanted to open a door.

The strategy relies on the matching asymptotic argument. Following the idea in \cite{BMR}, we aim at finding directly the relation \eqref{Lambda-theta} such that we can glue the outgoing solution at $r \rightarrow+\infty$ which is a small deformation of the singular solution, with the rescaled smooth solution which bifurcates from of the Emden-Fowler solution. The matching asymptotic strategy was implemented to construct special solutions like ground states (see \cite{budnor, DGW}) or self-similar profiles for the energy supercritical nonlinear heat equation(see \cite{CRS}). However, the implementation is not straightforward and it is in fact challenging. More precisely, we aim to construct the exterior solution $u_{\text{ext}}$ by perturbing the singular solution $\Phi$ but they correspond to two different eigenvalues $\lambda$ and $\lambda_*$ respectively. Since those eigenvalues have an impact on the asymptotic behavior of the corresponding eigenfunction at infinity, $u_{\text{ext}}$ cannot be obtained by simply perturbing $\Phi$. Instead, we work on $v_{\text{ext}}:=u_{\text{ext}}r^{-\frac{\varepsilon}{2}}$ where $\varepsilon:= \lambda-\lambda_*$. This correction removes the problem of the asymptotic at infinity and $\Phi$ as well as $v_{\text{ext}}$ corresponds to comparable eigenvalue problems. Although, it brings another difficulty while linearizing. In fact, we get a non-homogeneous term in the linearized equation which is unusual situation in the literature. We showed that this new term can be treated as an error comparing to $\Phi$ by carrefully analyzing all its component. The construction of the interior solution follows similar to \cite{budnor, DGW}. However, we should be careful with the impact of the additional nonlinearity $|u|^{q-1}u$. Also, the matching is similar to \cite{budnor, DGW}. Since the initial and the eigenvalue are arbitrary taken, the gluing will force the relation \eqref{Lambda-theta}.

Note that, the constructions of the exterior and interior solutions hold for any $q$ less than $p$ but the matching requires $q<\frac{p+1}{2}$ in order to keep the second therm in the asymptotic of $\Phi$ small during the gluing process.

This approach has an advantage of making the leading terms and the matching procedure appear more naturally and explicitly. It also allows a flexible functional framework in order to treat separately the neighborhood of the origin and the neighborhood of $\infty .$

In the second section, we start with constructing a set of exterior solutions. We first study some linear problems that are helpful for the construction argument. The set of exterior solutions is then obtained by a fixed point argument. In the third section, we construct a set of interior solutions that bifurcates from the solution to the Emden-Fowler equation. For that, we use a contraction mapping argument. The last section is dedicated to the matching argument using the Brower fixed point theorem. In the appendix, we sketch the proof of Theorem \ref{Thm-sing} and we investigate the behavior of the singular solution to \eqref{ODE-pq} at infinity.

\section{Exterior solution}
In this section, we are going to construct a solution which bifurcates from the singular solution $\Phi$. We denote by $\mathcal{L}$ the linearized operator around the singular solution $\Phi$ which is given by
\begin{equation}
\label{def-L}
\mathcal{L}:= -\partial_{rr}-\frac{d-1}{r}\partial_r +(r^2-\lambda_*)-p\Phi^{p-1}-q\Phi^{q-1}.
\end{equation}
In the following lemma, we give the behavior of the fundamental basis of its kernel.

\begin{lemma} \label{lem:ker-L}
There exists $\psi_1$ in the kernel of $\mathcal{L}$ with the following asymptotic behaviors:
\begin{equation}
\label{behv:infty}
\psi_1(r)= K e^{-\frac{r^2}{2}} r^{\frac{\lambda_*-d}{2}} \left(1+O\left(\frac{1}{r^2}\right)\right), \,\,  \,\, \psi_1'(r)= -K r e^{-\frac{r^2}{2}} r^{\frac{\lambda_*-d}{2}} \left(1+O\left(\frac{1}{r^2}\right)\right)
\end{equation}
as $r\to \infty$,
\begin{equation}\label{behv:origin}
\psi_1(r)= \frac{c_1\sin\left(\omega\log(r)+c_2\right)}{r^{\frac{d-2}{2}}}\left(1+O\left(r^{\frac{2(p-q))}{p-1}}\right)\right),\,\, \rm{as} \,\, r\to 0 ,
\end{equation}
and
\begin{equation}\label{behv:origin-deriv}
\psi_1'(r)= \left(\frac{\omega c_1\cos\left(\omega\log(r)+c_2\right)}{r^{\frac{d}{2}}}-\frac{d-2}{2}\frac{c_1\sin\left(\omega\log(r)+c_2\right)}{r^{\frac{d}{2}}}\right)\left(1+O\left(r^{\frac{2(p-q))}{p-1}}\right)\right),
\end{equation}
$\rm{as} \,\, r\to 0$ where $ \omega:=\sqrt{|\left(d-2\right)^2 -4 pA^{p-1}}|$, $c_1\neq 0$ and $c_2\in\mathbb{R}$.
Moreover, $\psi_2$, the second function in the the fundamental basis admits the following behaviors:
\begin{equation}
\label{behv:infty-2}
\psi_2=\frac{-1}{2K}  \frac{e^{\frac{r^2}{2}}}{r^{\frac{\lambda_*+d}{2}}}\left(1+O\left(\frac{1}{r^2}\right)\right)
\end{equation}
as $r\to \infty$ and
\begin{equation}\label{behv:origin-2}
\psi_2= \frac{c_3\sin\left(\omega\log(r)+c_4\right)}{r^{\frac{d-2}{2}}}\left(1+O\left(r^{\frac{2(p-q))}{p-1}}\right)\right)
\end{equation}
as $r\to 0$, where $ c_3\neq 0$ and $c_4\in\mathbb{R}$.
%there exists  $c\neq 0$ such that
%\be
%\label{refinedbeahviourpsione}
%\Lambda\psi_1=\frac{c}{r^{\frac{2}{p-1}+2}}\left(1+O\left(\frac{1}{r^2}\right)\right)\textrm{ as } r\to+\infty.
%\ee
\end{lemma}
\begin{proof}
\textbf{Step 1:} \underline{Behavior at infinity}.
The proof is similar to the one of Lemma \eqref{Lem:behav-Phi-infrty}. We only need to justify that $\psi_1$, a solution to $\mathcal{L}(\psi_1)=0$, does not change sign for $r$ sufficiently large. More precisely, if $\psi_1$ oscillates for $R$ large, then there exists $R$ such that $\psi_1(R)<0$, $\psi_1'(R)=0$ and $\psi_1''(R)>0$. This way, from the equation $\mathcal{L}(\psi_1)=0$, we have
$\psi_1''(R)= (R^2-\lambda_*-p\Phi^{p-1}(R)-q\Phi^{q-1}(R))\psi_1(R)<0$ which is a contradiction. This means that we can take $\psi_1$ a positive (for $r$ large) solution to $\mathcal{L}(\psi_1)=0$ which decays to zero when $r$ goes to infinity. This is sufficient to obtain \eqref{behv:infty}.
%For $\psi_2$, we use the reduction of order formula
%$$\psi_2=\int \frac{\boW(s)}{\psi_1^2(s)}ds\psi_1 $$
%where $\boW(r)$ is the Wronskian\footnote{We recall that the Wronskian of two functions $f$ and $g$ is defined by $\boW(f,g):=f'g-g'f$.} of $\psi_1$ and $\psi_2$. It is is given by
%\begin{equation}
%\label{Wronskian:ODEpsi}
%\boW(r)=r^{-(d-1)}.
%\end{equation}

\textbf{Step 2:} \underline{Behavior near the origin}.
In this subsection, we take $r$ in a neighbourhood of $0$. We consider the following second order ODE
\begin{align}
\label{ODE-1}
\phi''(r)+\frac{d-1}{r}\phi'(r)+ p\frac{A^{p-1}}{r^2} \phi(r)=0,
\end{align}
The fundamental basis of solutions to the homogeneous equation associated to \eqref{ODE-1} is given by
\begin{equation}
\label{phi1-phi2}
\phi_1(r)= \frac{\sin(\omega \log(r))}{r^{\frac{d-2}{2}}} \quad {\rm and} \quad  \phi_2(r)=\frac{\cos(\omega \log(r))}{r^{\frac{d-2}{2}}},
\end{equation}
with $ \omega:=\sqrt{|\left(d-2\right)^2 -4 pA^{p-1}}|$.
%$$m_\pm:=-\frac{d-2}{2}\pm \frac{\omega}{2} ,$$
%solve the following algebraic equation
%$$m^2 +\left(d-2\right)m + pA^{p-1}=0$$
%where the discriminant is given by
%\begin{align*}
%\omega^2= \left(d-2\right)^2 -4 pA^{p-1}.
%\end{align*}
%{\color{red} \begin{verbatim}
%I am not sure about the sign of this discriminant. However, I checked
%that it is negative for any dimension less or equal to 10.
%So that I think its sign is related to the Joseph-Lundgren exponent.
%\end{verbatim}
%}
Note that, the Wronskian\footnote{We recall that the Wronskian of two differentiable functions $f$ and $g$ is defined by $\boW(f,g):=f'g-g'f$.} of $\phi_1$ and $\phi_2$ is given by
\begin{equation}
\label{Wronskian:ODE1}
\boW(r)= \omega r^{-(d-1)}.
\end{equation}
%{\color{red} \begin{remark}
%Collot, Raphael and Szeftel chose $\gamma$ that solves $\gamma^2-(d-2)\gamma+pA^{p-1}=0.$ Note that this choice will lead to the same behavior at the end.
%\end{remark}
%}

Using the variation of constants, $\psi_1$ is given by
$$\psi_1(r)=\left(a_{1,0}+\omega\int_{0}^{r} f \psi_1 \phi_{2} s^{d-1} d s\right) \phi_{1}+\left(a_{2,0}-\omega\int_{0}^{r} f \psi_1 \phi_{1} s^{d-1} d s \right) \phi_{2}
$$
where
\begin{equation}
\label{def:f}
f(r):=-\lambda_* +r^2 -p\left(\Phi^{p-1} -  \frac{A^{p-1}}{r^2} \right)-q\Phi^{q-1}.
\end{equation}

We write

$$\psi_1 = a_{1,0} \phi_1 + a_{2,0} \phi_2 + \omega \tilde{\phi}, \quad \tilde{\phi}= \boF\left(\tilde{\phi}\right)$$
where
\begin{align*}
\boF\left(\tilde{\phi}\right) =& \int_{0}^{r} f \left( a_{1,0} \phi_1 + a_{2,0} \phi_2 + \tilde{\phi} \right) \phi_{2} s^{d-1} d s \phi_1\\
&-\int_{0}^{r} f \left( a_{1,0} \phi_1 + a_{2,0} \phi_2 + \tilde{\phi} \right)  \phi_{1} s^{d-1} d s  \phi_{2}
\end{align*}

From \eqref{behav:phi:origin}, we have

$$f(r)=-\lambda_* +r^2 -  \frac{A^{p-1}}{r^2} \left(\left\{1+\boO\left(r^{\frac{2(p-q))}{p-1}}\right)\right\}^{p-1} - 1  \right) - \boO\left(r^{-\frac{2(q-1)}{p-1}}\right).$$

Since $r^{\frac{2(p-q)}{p-1}} \to 0$ as $r\to 0$,  we write

\begin{align*}
\left\{1+\boO\left(r^{\frac{2(p-q))}{p-1}}\right)\right\}^{p-1} - 1 &= \exp\left((p-1)\ln\left\{1+\boO\left(r^{\frac{2(p-q))}{p-1}}\right)\right\} \right) - 1\\
& = \sum_{k\geq1} \boO\left(r^{\frac{2k(p-q))}{p-1}}\right)\\
& = \boO\left(r^{\frac{2(p-q))}{p-1}}\right)
\end{align*}

So that,
$$f(r)=-\lambda_* +r^2 -  r^{-2}  \boO\left(r^{\frac{2(p-q))}{p-1}}\right) .$$

%$$\frac{2(q-1)}{p-1}= 2-\frac{2(p-q)}{p-1}$$

Hence,

\begin{align*}
r^{\frac{d-2}{2}}\left|\boF\left(\tilde{\phi}\right) \right| \lesssim & \int_{0}^{r} s^{\frac{2(p-q))}{p-1}-1} ds + \int_{0}^{r} s^{\frac{d}{2}} |\tilde{\phi}| ds\\
\lesssim & r^{\frac{2(p-q))}{p-1}} + \int_{0}^{r} s^{\frac{d}{2}} |\tilde{\phi}| ds
\end{align*}

Thus, let $r_0>0$ small enough, the Banach fixed point theorem applies in the space equipped with the norm

$$\sup_{r<r_0} r^{\frac{d-2}{2}} |\tilde{\phi}|(r). $$

For $\psi_2$, we use the reduction of order formula
$$\psi_2=\int \frac{\boW(s)}{\psi_1^2(s)}ds \ \psi_1 $$
where $\boW(r)$ is the Wronskian of $\psi_1$ and $\psi_2$. It is is given by
\begin{equation}
\label{Wronskian:ODEpsi}
\boW(r)=r^{-(d-1)}.
\end{equation}
This finishes the proof of this lemma.
\end{proof}

\begin{remark}
The condition $p<p_{\rm JL}$ is very important to get \eqref{phi1-phi2}. In this case the quadratic formula of the characteristic equation to \eqref{ODE-1} is negative i.e. $\left(d-2\right)^2 -4 pA^{p-1} <0$. 
\end{remark}
%Including \eqref{phi1-phi2} and \eqref{def:f} into \eqref{eq:alpha1}, we get
%$$|\alpha_1| \lesssim \|\alpha\|  \int_0^{r} \left(\lambda +r^2+ r^{2-\frac{2}{p-1}}\right) s^{\frac{d}{2}} \,ds \lesssim \|\alpha\|  r^{\frac{d+2}{2}} \left( 1 + r^{2-\frac{2}{p-1}}\right).$$

Let $\varepsilon$ be an arbitrary number in a neighborhood of $0$. We set
\begin{equation}
\label{def:tilde-d-lambda}
\tilde{d}=d+\varepsilon \quad \texttt{and} \quad \lambda=\lambda_*+\varepsilon.
\end{equation}
  We shall study the behavior of $\tilde{\psi}_1$ and $\tilde{\psi}_2$, the functions in the fundamental set of solutions to
\begin{equation}
\label{ODE-psi-tilde}
\partial_{rr}\tilde{\psi}+\frac{\tilde{d}-1}{r} \partial_{r}\tilde{\psi} -( r^2 -\lambda ) \tilde{\psi} + p \Phi^{p-1}\tilde{\psi}+ q \Phi^{q-1}\tilde{\psi}=0, \quad r>0.
\end{equation}
Following the strategy of Lemma \ref{lem:ker-L}, we have the following result.
\begin{lemma}
There exists two solutions $\tilde{\psi}_1$ and $\tilde{\psi}_2$ of \eqref{ODE-psi-tilde} with the following asymptotic behaviors:
\begin{equation}
\label{behv:infty-tilde}
\tilde{\psi}_1= K e^{-\frac{r^2}{2}} r^{\frac{\lambda_*-d}{2}} \left(1+O\left(\frac{1}{r^2}\right)\right), \,\,\tilde{\psi}_2=\frac{-1}{2K}  \frac{e^{\frac{r^2}{2}}}{r^{\frac{\lambda+\tilde{d}}{2}}}\left(1+O\left(\frac{1}{r^2}\right)\right)
\end{equation}
as $r\to \infty$ and
\begin{align}\label{behv:origin-tilde}
\tilde{\psi}_1&= \frac{\tilde{c}_1\sin\left(\tilde{\omega}\log(r)+\tilde{c}_2\right)}{r^{\frac{\tilde{d}-2}{2}}}\left(1+O\left(r^{\frac{2(p-q))}{p-1}}\right)\right),\\
\tilde{\psi}_2&= \frac{\tilde{c}_3\sin\left(\tilde{\omega}\log(r)+\tilde{c}_4\right)}{r^{\frac{\tilde{d}-2}{2}}}\left(1+O\left(r^{\frac{2(p-q))}{p-1}}\right)\right) \nonumber
\end{align}
as $r\to 0$, where $ \tilde{\omega}=\sqrt{|\left(\tilde{d}-2\right)^2 -4 pA^{p-1}}|$, $\tilde{c}_1, \tilde{c}_3\neq 0$ and $\tilde{c}_2, \tilde{c}_4\in\mathbb{R}$.
%Moreover, there exists  $c\neq 0$ such that
%\be
%\label{refinedbeahviourpsione}
%\Lambda\psi_1=\frac{c}{r^{\frac{2}{p-1}+2}}\left(1+O\left(\frac{1}{r^2}\right)\right)\textrm{ as } r\to+\infty.
%\ee
\end{lemma}

\begin{proof}
The proof follows exactly as in the one of Lemma \eqref{lem:ker-L}.
%For $\tilde{\psi}_2$, we use the reduction of order formula
%$$\tilde{\psi}_2=\int \frac{\boW(s)}{\tilde{\psi}_1^2(s)}ds\tilde{\psi}_1 $$
%where $\boW(r)$ is the Wronskian\footnote{We recall that the Wronskian of two functions $f$ and $g$ is defined by $\boW(f,g):=f'g-g'f$.} of $\tilde{\psi}_1$ and $\tilde{\psi}_2$. It is is given by
%\begin{equation}
%\label{Wronskian:ODEpsi}
%\boW(r)=r^{-(\tilde{d}-1)}.
%\end{equation}

\end{proof}

Let $0<r_*<1$. We denote by $X_{r_*}$ the space of functions on $[r_*,+\infty)$ endowed with the following norm
$$\|w\|_{X_{r_*}}=\sup_{r_*\leq r\leq 1} \left\lbrace r^{\frac{d-2}{2}}|w|\right\rbrace +\sup_{r\geq 1}\left\{  e^{\frac{r^2}{2}} r^{2-\frac{\lambda_*-d}{2}}|w| \right\}.$$
We define the resolvent
\begin{equation}
\label{resol-L}
\tilde{\boT}(f) := \left(\int_r^{+\infty} f\tilde{\psi}_2 s^{\tilde{d}-1}ds\right)\tilde{\psi}_1- \left(\int_r^{+\infty} f\tilde{\psi}_1{s}^{\tilde{d}-1}ds\right)\tilde{\psi}_2,
\end{equation}
i.e.
$$\tilde{\mathcal{L}}(\tilde{\boT}(f))=f,$$
where $\tilde{\mathcal{L}}$ is the linearized operator around the singular solution $\Phi$, given by
\begin{equation}
\label{def-L-tilde}
\tilde{\mathcal{L}}:= -\partial_{rr}-\frac{\tilde{d}-1}{r}\partial_r +(r^2-\lambda)-p\Phi^{p-1}-q\Phi^{q-1}.
\end{equation}

We assume that
\begin{equation}
\label{cond-delta}
0<|\varepsilon| \ll \frac{1}{|\log(r_*)|}  .
\end{equation}
The following lemma states the continuity of the resolvent in the space $X_{r_*}$.
\begin{lemma}
\label{lem:cont-boT}
For any $f\in X_{r_*}$, we have
\begin{equation}
\label{cont-resolvent}
\|\tilde{\boT}(f)\|_{X_{r_*}}\lesssim \int_{r_*}^1 |f|{s}^{\frac{\tilde{d}}{2}}ds + \sup_{r\geq 1}\left\{  e^{\frac{r^2}{2}} r^{2-\frac{\lambda_*-d}{2}}|f| \right\}.
\end{equation}

\end{lemma}

\begin{proof}
First, we consider the case when $r\geq 1$. Using \eqref{def:tilde-d-lambda}, \eqref{behv:infty-tilde} and \eqref{resol-L}, we obtain
\begin{align*}
|\tilde{\boT}(f)|&\lesssim  \int_r^{+\infty} |f| e^{\frac{s^2}{2}} {s}^{-\frac{\lambda+\tilde{d}}{2}} s^{\tilde{d}-1} ds \, e^{-\frac{r^2}{2}} {r}^{\frac{\lambda_*-d}{2}} +\int_r^{+\infty} |f|e^{-\frac{{s}^2}{2}} {s}^{\frac{\lambda_*-d}{2}} {s}^{\tilde{d}-1}ds\, e^{\frac{r^2}{2}} {r}^{-\frac{\lambda+\tilde{d}}{2}}\\
& \lesssim  \int_r^{+\infty} |f| e^{\frac{s^2}{2}} {s}^{\frac{\tilde{d}-\lambda}{2}-1}  ds \, e^{-\frac{r^2}{2}} {r}^{\frac{\lambda_*-d}{2}} +\int_r^{+\infty} |f|e^{-\frac{{s}^2}{2}} {s}^{\frac{\lambda_*-d}{2}} {s}^{\tilde{d}-1}ds\, e^{\frac{r^2}{2}} {r}^{-\frac{\lambda+\tilde{d}}{2}} \\
&\lesssim \left\{\int_r^{+\infty}\frac{ds}{{s}^3} e^{-\frac{r^2}{2}} {r}^{\frac{\lambda_*-d}{2}} +e^{\frac{r^2}{2}}{r}^{-\frac{\lambda+\tilde{d}}{2}} \int_r^{+\infty} {s}^{\lambda-3}e^{-{s}^2}ds\right\}\sup_{r\geq 1}\left\{  e^{\frac{r^2}{2}} r^{2-\frac{\lambda_*-d}{2}}|f| \right\}\\
&\lesssim e^{-\frac{r^2}{2}} {r}^{\frac{\lambda_*-d}{2}-2} \sup_{r\geq 1}\left\{  e^{\frac{r^2}{2}} r^{2-\frac{\lambda_*-d}{2}}|f| \right\},
\end{align*}
where we have used in the last step that the function $s\mapsto {s}^{\lambda}e^{-{s}^2}$ is decreasing for $s$ large.

On the other hand, for $r_*\leq r\leq 1$, similarly we infer from \eqref{behv:infty-tilde}, \eqref{behv:origin-tilde} and \eqref{cond-delta}, that
\begin{align*}
r^{\frac{d-2}{2}}|\tilde{\boT}(f)|& \lesssim  r^{-\frac{\varepsilon}{2}} \left( \int_1^{+\infty}\frac{ds}{{s}^3} + \int_1^{+\infty} {s}^{\lambda-3}e^{-{s}^2}ds\right)  \sup_{r\geq 1}\left\{  e^{\frac{r^2}{2}} r^{2-\frac{\lambda-d}{2}}|f| \right\}\\
&\qquad +\int_r^1 |f|{s}^{\frac{\tilde{d}}{2}}ds\ r^{-\frac{\varepsilon}{2}} \\
&\lesssim \int_{r_*}^1 |f|{s}^{\frac{\tilde{d}}{2}}ds +  \sup_{r\geq 1}\left\{  e^{\frac{r^2}{2}} r^{2-\frac{\lambda-d}{2}}|f| \right\}.
\end{align*}
This finishes the proof of this lemma.
\end{proof}
%In addition, we define the resolvent of the operator $\mathcal{L}$ which is the linearized operator around the singular solution $\Phi$ and it is defined by \eqref{def-L},
%\begin{equation}
%\label{resol-L}
%\tilde{\boT}(f) := \left(\int_r^{+\infty} f \tilde{\psi}_2 s^{d-1}ds\right)\tilde{\psi}_1- \left(\int_r^{+\infty} f\tilde{\psi}_1{s}^{d-1}ds\right)\tilde{\psi}_2,
%\end{equation}
%i.e.
%$$\tilde{\mathcal{L}}(\tilde{\boT}(f))=f,$$
%Similarly, we claim the continuity of $\boT$ in $X_{r_*}$.
%\begin{lemma}
%\label{lem:cont-boT}
%For any $f\in X_{r_*}$, we have
%\begin{equation}
%\label{cont-resolvent-T}
%\|\tilde{\boT}(f)\|_{X_{r_*}}\lesssim \int_{r_*}^1 |f|{r'}^{\frac{d}{2}}ds + \sup_{r\geq 1}\left\{  e^{\frac{r^2}{2}} r^{2-\frac{\lambda_*-d}{2}}|w| \right\}.
%\end{equation}
%
%\begin{remark}
%This lemma can be seen as a direct consequence of Lemma \ref{lem:cont-boT-tilde} when we take $\varepsilon=0.$
%\end{remark}
%
%\end{lemma}

\begin{lemma}
We denote by $\psi$ the decaying solution to the following non homogeneous ODE
$$\tilde{\mathcal{L}}(\psi)=  \frac{p-1}{2} \log(r)  \Phi^p + \frac{q-1}{2} \log(r)  \Phi^q + \left(\frac{\varepsilon}{4} + \frac{d}{2} -1 \right) \frac{\Phi}{r^2} -\left(\frac{\Phi'}{r}+\Phi\right).$$
Its asymptotic behavior is given by:
\begin{equation}
\label{behv:infty-psi}
\psi= K e^{-\frac{r^2}{2}} r^{\frac{\lambda_*-d}{2}} \left(1+O\left(\frac{1}{r^2}\right)\right),
\end{equation}
as $r\to \infty$ and \footnote{Here $\sigma :=\frac{d}{2}-\frac{2}{p-1}$ is the critical sobolev exponent.}
\begin{equation}\label{behv:origin-psi}
\psi = \frac{\tilde{c}_1\sin\left(\tilde{\omega}\log(r)+\tilde{c}_2\right)}{r^{\frac{\tilde{d}-2}{2}}}\left(1+O\left( |\log(r)|  r^{\sigma -1}\right)+O\left(  r^{\sigma -1}\right)\right),
\end{equation}
as $r\to 0$, where $ \tilde{\omega}=\sqrt{|\left(\tilde{d}-2\right)^2 -4 pA^{p-1}}|$, $\tilde{c}_1, \tilde{c}_3\neq 0$ and $\tilde{c}_2, \tilde{c}_4\in\mathbb{R}$.

\end{lemma}

\begin{proof}
We write
$$ \psi= a_1 \tilde{\psi}_1 + b_2 \tilde{\psi}_2 + \tilde{\psi}$$
where
$$\tilde{\psi}=\tilde{\boT}\left(  \frac{p-1}{2}\log(r) \Phi^p + \frac{q-1}{2} \log(r)  \Phi^q  + \left(\frac{\varepsilon}{4} + \frac{d}{2} -1 \right) \frac{\Phi}{r^2} -\left(\frac{\Phi'}{r}+\Phi\right)\right).$$

For $r\geq 1$, from \eqref{behv:infty-tilde} and  \eqref{resol-L}, we have
\begin{align*}
|\tilde{\boT}( \log(r) \Phi^p )|&\lesssim  \int_r^{+\infty}  \log(s) \ e^{(1-p)\frac{s^2}{2}}  {s}^{p\frac{\lambda_*-d}{2}} {s}^{-\frac{\lambda+\tilde{d}}{2}} s^{\tilde{d}-1} ds \, e^{-\frac{r^2}{2}} {r}^{\frac{\lambda_*-d}{2}} \\
& \quad + \int_r^{+\infty} \log(s) \ e^{-(p+1)\frac{{s}^2}{2}} {s}^{(p+1)\frac{\lambda_*-d}{2}} {s}^{\tilde{d}-1}ds\, e^{\frac{r^2}{2}} {r}^{-\frac{\lambda+\tilde{d}}{2}}\\
&\lesssim \frac{1}{r^2} \tilde{\psi}_1
\end{align*}

\begin{align*}
|\tilde{\boT}( \log(r) \Phi^q )|&\lesssim  \int_r^{+\infty}  \log(s) \ e^{(1-q)\frac{s^2}{2}}  {s}^{q\frac{\lambda_*-d}{2}} {s}^{-\frac{\lambda+\tilde{d}}{2}} s^{\tilde{d}-1} ds \, e^{-\frac{r^2}{2}} {r}^{\frac{\lambda_*-d}{2}} \\
& \quad + \int_r^{+\infty} \log(s) \ e^{-(q+1)\frac{{s}^2}{2}} {s}^{(q+1)\frac{\lambda_*-d}{2}} {s}^{\tilde{d}-1}ds\, e^{\frac{r^2}{2}} {r}^{-\frac{\lambda+\tilde{d}}{2}}\\
&\lesssim \frac{1}{r^2} \tilde{\psi}_1
\end{align*}

and
\begin{align*}
\left|\tilde{\boT}\left(  \frac{\Phi}{r^2} \right)\right|+\left|\tilde{\boT}\left(  \frac{\Phi'}{r}+\Phi \right)\right|&\lesssim  \int_r^{+\infty}   s^{-3} ds \, e^{-\frac{r^2}{2}} {r}^{\frac{\lambda_*-d}{2}} + \int_r^{+\infty} e^{-{s}^2} {s}^{\lambda_*-d} {s}^{\tilde{d}-3}ds\, e^{\frac{r^2}{2}} {r}^{-\frac{\lambda+\tilde{d}}{2}}\\
&\lesssim \frac{1}{r^2} \tilde{\psi}_1 .
\end{align*}
Note that, in the last estimate we have used
\begin{equation}
\label{bound:infty:phi'/r+phi}
\left|\frac{\Phi'}{r}+\Phi\right| \lesssim \frac{1}{r^2} \left( e^{-\frac{r^2}{2}} r^{\frac{\lambda_*-d}{2}} \right) ,
\end{equation}

which is a direct consequence of \eqref{behv:infty:Phi}.
On the other hand, for $r_*<r<1$, similarly we infer from \eqref{behv:infty-tilde}, \eqref{behv:origin-tilde} and \eqref{cond-delta}, that
\begin{align*}
|\tilde{\boT}(\log(r) \  \Phi^p)|& \lesssim   \left( \int_1^{+\infty}\frac{ds}{{s}^3} + \int_1^{+\infty} {s}^{\lambda-3}e^{-{s}^2}ds\right) |\tilde{\psi}_1+\tilde{\psi}_2|\\
&\qquad +\int_{r}^1 |\log s|  {s}^{\sigma -2}ds |\tilde{\psi}_1+\tilde{\psi}_2|\\
&\lesssim \left(1+ |\log r| {r}^{\sigma -1}  + {r}^{\sigma -1}  \right)  |\tilde{\psi}_1+\tilde{\psi}_2|,
\end{align*}
and
\begin{align*}
|\tilde{\boT}(\log(r) \  \Phi^q)|& \lesssim   \left( \int_1^{+\infty}\frac{ds}{{s}^3} + \int_1^{+\infty} {s}^{\lambda-3}e^{-{s}^2}ds\right) |\tilde{\psi}_1+\tilde{\psi}_2|\\
&\qquad +\int_{r}^1 |\log s|  {s}^{\sigma -2} s^{\frac{2(p-q))}{p-1}}ds |\tilde{\psi}_1+\tilde{\psi}_2|\\
&\lesssim \left(1+ |\log r| {r}^{\sigma -1} r^{\frac{2(p-q))}{p-1}}+ {r}^{\sigma -1} r^{\frac{2(p-q))}{p-1}}\right)  |\tilde{\psi}_1+\tilde{\psi}_2|.
\end{align*}
Here we have used the following identity which implement integration by parts
\begin{align*}
(\alpha+1)\int_{r}^1 \log (s ) {s}^{\alpha}ds = \log(r) r^{\alpha+1} - \int_{r}^1 {s}^{\alpha}ds =  \log(r) r^{\alpha+1} - 1 +\frac{{r}^{\alpha}}{\alpha+1} , \quad \alpha>-1.
\end{align*}

Similarly, we have
\begin{align*}
\left|\tilde{\boT}\left(  \frac{\Phi}{r^2} \right)\right|+\left|\tilde{\boT}\left(  \frac{\Phi'}{r}+\Phi \right)\right|& \lesssim   \left( \int_1^{+\infty}\frac{ds}{{s}^3} + \int_1^{+\infty} {s}^{\lambda-3}e^{-{s}^2}ds\right) |\tilde{\psi}_1+\tilde{\psi}_2|\\
&\qquad +\int_{r}^1   {s}^{\sigma -2}ds |\tilde{\psi}_1+\tilde{\psi}_2|\\
&\lesssim \left(1+  {r}^{\sigma -1}\right)  |\tilde{\psi}_1+\tilde{\psi}_2|.
\end{align*}
This finishes the proof of this lemma.

\end{proof}

With all those lemmas in hand, we are able to construct a set of exterior solutions.
\begin{proposition}
\label{prop:ext-sol}
Let $0<r_*<1$ a small enough universal constant and $\varepsilon:=\lambda-\lambda_*$ be an arbitrary parameter satisfying
\begin{equation}
\label{est:r0-eps}
0<|\varepsilon|\ll  r_*^{\sigma-1} ,
\end{equation}
%\frac{}{|\log r_*|^2 } ,
there exists a postive solution $u$ to \eqref{ODE-pq}, for $\lambda=\lambda_*+\varepsilon$, of the form
\begin{equation}
\label{decomp:ext}
u=r^{\frac{\varepsilon}{2}} \left( \Phi+\varepsilon(\psi+ w)\right)
\end{equation}
with the bound:
\begin{equation}
\label{bound:norm-w}
\|w\|_{X_{r_*}}+\left\| \inf\{1,r\}(1+r)^{-1} w'\right\|_{X_{r_*}}\lesssim \varepsilon r_*^{1-\sigma}.
\end{equation}
Moreover, we have $w_{|_{\varepsilon=0}}=0$ and
\begin{equation}
\label{estim-deriv-w-ep}
 \left\|\partial_\varepsilon w_{|_{\varepsilon=0}}\right\|_{X_{r_*}}\lesssim   r_*^{1-\sigma}.
\end{equation}

\end{proposition}

\begin{proof}
The proof relies on the Banach fixed point theorem. First, we denote by
$$v:=r^{-\frac{\varepsilon}{2}} u.$$
It is a solution to the following equation
\begin{equation}
\label{ODE:v}
v''+\frac{\tilde{d}-1}{r} v' -( r^2-\lambda ) v + r^{\varepsilon\frac{p-1}{2}}  |v|^{p-1}v +  r^{\varepsilon\frac{q-1}{2}}  |v|^{q-1}v =\frac{\varepsilon}{2} \left(\frac{\varepsilon}{2} + d -2 \right) \frac{v}{r^2}  , \quad r>r_*,
\end{equation}
where $\tilde{d}=d+\varepsilon$. It remains to construct $v$ of the form $v=\Phi+\varepsilon(\psi+ w )$. By definition, $v$ is positive. Additionally, from \eqref{est:r0-eps} and the fact that $\|w\|_{X_{r_*}}\leq 1$, we have
\begin{equation}
\label{pos-u}
 2\Phi \geq v \geq \Phi - \varepsilon |\psi+w| \geq \frac{1}{2} \Phi \textrm{ for }r\leq r_*.
\end{equation}
We plug the decomposition form of $v$ into \eqref{ODE:v}, we obtain that $w$  satisfies the following identity
$$\tilde{\mathcal{L}}(w)= \varepsilon F(\Phi,w,\varepsilon)\textrm{ on } [r_*,\infty)$$
where $\tilde{\mathcal{L}}$ is given by \eqref{def-L-tilde} and
\begin{align*}
F(\Phi,w,\varepsilon) = &\frac{1}{\varepsilon^2} \left((\Phi+\varepsilon (\psi+w) )^p-\Phi^p-p\Phi^{p-1}\varepsilon (\psi+w)  \right)\\
&+\frac{1}{\varepsilon^2} \left((\Phi+\varepsilon (\psi+w) )^q-\Phi^q-q\Phi^{q-1}\varepsilon (\psi+w)  \right)\\
&+ \frac{1}{\varepsilon^2}(1-r^{\varepsilon\frac{p-1}{2}} )  \left((\Phi+\varepsilon (\psi+w) )^p-\Phi^p \right)+\\
&+ \frac{1}{\varepsilon^2}(1-r^{\varepsilon\frac{p-1}{2}} + \frac{p-1}{2} \varepsilon \log r  )\Phi^p  - \left(\frac{\varepsilon}{4} + \frac{d}{2} -1 \right) \frac{\psi+w}{r^2}.
\end{align*}

We remark that $F(\Phi,w,\varepsilon)$ can be rewritten as
\begin{align}
\label{def2:F}
F(\Phi,w,\varepsilon) =& p(p-1)\left(\int_0^1(1-t)(\Phi+t\varepsilon (\psi+w) )^{p-2}dt\right) (\psi+w) ^2 \\
& + q(q-1)\left(\int_0^1(1-t)(\Phi+t\varepsilon (\psi+w) )^{q-2}dt\right) (\psi+w) ^2 \nonumber\\
&+ \frac{p}{\varepsilon}(1-r^{\varepsilon\frac{p-1}{2}} )  \left(\int_0^1(\Phi+t\varepsilon (\psi+w) )^{p-1}dt\right) (\psi+w)\nonumber\\
&+ \frac{q}{\varepsilon}(1-r^{\varepsilon\frac{q-1}{2}} )  \left(\int_0^1(\Phi+t\varepsilon (\psi+w) )^{q-1}dt\right) (\psi+w)\nonumber\\
&+ \frac{1}{\varepsilon^2}(1-r^{\varepsilon\frac{p-1}{2}} +\frac{p-1}{2} \varepsilon \log r  )\Phi^p + \left(\frac{\varepsilon}{4} + \frac{d}{2} -1 \right) \frac{\psi+w}{r^2}\nonumber\\
&+ \frac{1}{\varepsilon^2}(1-r^{\varepsilon\frac{q-1}{2}} +\frac{q-1}{2} \varepsilon \log r  )\Phi^q . \nonumber
\end{align}

We claim that
\begin{equation}
\label{nonbound1}
\int_{r_*}^1|F(\Phi,w,\varepsilon)|{s}^{\frac{ \tilde{d} }{2}}ds + \sup_{r\geq 1} \left\lbrace e^{\frac{r^2}{2}} r^{2-\frac{\lambda-d}{2}} |F(\Phi,v,\varepsilon)| \right\rbrace \lesssim  r_*^{1-\sigma}
\end{equation}
and
\begin{align}
\label{nonbound2}
\nonumber & \int_{r_*}^1|F(\Phi,w_1,\varepsilon)-F(\Phi,w_2,\varepsilon)|{s}^{\frac{\tilde{d}}{2}}ds + \sup_{r\geq 1}\left\lbrace e^{\frac{r^2}{2}} r^{2-\frac{\lambda-d}{2}}  |F(\Phi,w_1,\varepsilon)-F(\Phi,w_2,\varepsilon)| \right\rbrace\\
& \lesssim  r_*^{1-\sigma}\|w_1 - w_2\|_{X_{r_*}}.
\end{align}
Indeed, from \eqref{pos-u} and \eqref{def2:F}, we write
\begin{align}
\label{estim:F}
|F(\Phi,w,\varepsilon)| \lesssim &\Phi^{p-2} (\psi+w) ^2 + |\log r|    \ r^{\pm|\varepsilon|\frac{p-1}{2}}\Phi^{p-1} |\psi+w| \\
&+\Phi^{q-2} (\psi+w) ^2 + |\log r|    \ r^{\pm|\varepsilon|\frac{q-1}{2}}\Phi^{q-1} |\psi+w| \nonumber \\
&+ |\log r|^2    \ r^{\pm|\varepsilon|\frac{p-1}{2}}\Phi^{p} + \frac{|\psi_1+w|}{r^2}  + |\log r|^2    \ r^{\pm|\varepsilon|\frac{q-1}{2}}\Phi^{q}   .\nonumber
\end{align}
Here we used the following series expansion
\begin{align*}
|r^{\varepsilon\frac{p-1}{2}} -1| &= \sum_{n\geq1} \frac{(\log r)^n (p-1)^n \varepsilon^{n}}{n!2^n} = \varepsilon\frac{ (p-1) }{2}\log r\sum_{n\geq1} \frac{(\log r)^{n-1} (p-1)^{n-1} \varepsilon^{n-1}}{n \cdot (n-1)!\ 2^{n-1}}\\
& \leq|\varepsilon|\frac{ (p-1) }{2}|\log r| \sum_{n\geq0} \frac{|\log r|^n (p-1)^n |\varepsilon|^{n}}{n!2^n} \\
& \leq|\varepsilon|\frac{ (p-1) }{2}|\log r|   \ r^{\pm|\varepsilon|\frac{p-1}{2}}.
\end{align*}
Similarly, we have
\begin{align*}
|r^{\varepsilon\frac{q-1}{2}} -1| \leq|\varepsilon|\frac{ q-1 }{2}|\log r|   \ r^{\pm|\varepsilon|\frac{q-1}{2}}.
\end{align*}
Thus, for $r\in[r_*,1]$, from \eqref{behav:phi:origin}, \eqref{behv:origin}, \eqref{est:r0-eps} and \eqref{pos-u}, we have
$$|F(\Phi,w,\varepsilon)| \lesssim r^{\frac{2}{p-1}-d} + |\log r| \ r^{-\frac{d}{2}-1} + |\log r|^2 \ r^{-\frac{2}{p-1}-2} + r^{-\frac{d}{2}-1} $$
so that
\begin{equation}
\label{nonbound1-1}
\int_{r_*}^1|F(\Phi,v,\varepsilon)|{s}^{\frac{\tilde{d}}{2}}ds \lesssim \left( r_*^{1-\sigma} + |\log r_*|^2   r_*^{\sigma-1} + |\log r_*| \right) \left( 1+ r^{\frac{2(p-q)}{p-1}} \right) .
\end{equation}
For $r\geq 1$, we obtain similarly
\begin{align*}
|F(\Phi,w,\varepsilon)|  \lesssim & \left(1+\log r  \ r^{|\varepsilon|\frac{p-1}{2}}+ (\log r)^2  \ r^{|\varepsilon|\frac{p-1}{2}} \right)\Phi^{p} +\frac{\Phi }{r^2} + \left|\frac{\Phi'}{r}+\Phi\right|\\
&+ \left(1+\log r  \ r^{|\varepsilon|\frac{q-1}{2}}+ (\log r)^2  \ r^{|\varepsilon|\frac{q-1}{2}} \right)\Phi^{q}.
\end{align*}
Hence, using \eqref{behv:infty:Phi} and \eqref{bound:infty:phi'/r+phi}, we get
\begin{equation}
\label{nonbound1-2}
\sup_{r\geq 1} \left\lbrace e^{\frac{r^2}{2}} r^{2-\frac{\lambda_*-d}{2}} |F(\Phi,w,\varepsilon)| \right\rbrace \lesssim 1.
\end{equation}
Combining \eqref{nonbound1-1} with \eqref{nonbound1-2}, we conclude \eqref{nonbound1}. The proof of \eqref{nonbound2} is similar.
%since
%$$|F(\Phi,w,\varepsilon)| \lesssim \Phi^{p-2} \left|w_1^2-w_2^2 \right| \lesssim  \Phi^{p-2} \left(|v_1|+|v_2| \right) \left|w_1-w_2 \right|.$$
Next, we recall that $w$ is a solution to the following fixed point problem
\begin{equation}
\label{eq:Banachproblem}
w= \varepsilon \tilde{\boT}\Big(F(\Phi,w,\varepsilon)\Big),\,\,\,\,  w\in X_{r_*}.
\end{equation}

%With \eqref{nonbound1} and \eqref{nonbound2} at hand, we shall finish the proof of this proposition.
%Assume \eqref{firsbound}, \eqref{seconbound}, then we look for $w$ as the solution of the following fixed point

Combining the continuity estimate on the resolvent \eqref{cont-resolvent} with the nonlinear estimates \eqref{nonbound1} and \eqref{nonbound2}, we infer using the Banach fixed point theorem that there exists a unique solution $w$ to \eqref{eq:Banachproblem} .

We compute the derivative of the right hand side of \eqref{eq:Banachproblem} with respect to $r$ and we infer in a similar way the estimate of the second term in the right-hand side of \eqref{bound:norm-w}.

Finally, we compute $w_{|_{\varepsilon =0}}$ and $\partial_\varepsilon w_{|_{\varepsilon =0}}$. Note that $w_{|_{\varepsilon =0}}=0$ is a direct consequence of taking $\varepsilon=0$ in \eqref{eq:Banachproblem}. Furthermore, we differentiate \eqref{eq:Banachproblem} with respect to $\varepsilon$ and we evaluate it at $\varepsilon=0$ in order to obtain
%$$\partial_\varepsilon  w=p(p-1) \TT\Big(G[\Phi_*,\psi, \varepsilon ]w\Big)+\varepsilon  p(p-1)  \TT\Big(\partial_\varepsilon  G[\Phi_*,\psi, \varepsilon ]w\Big)$$
%and hence
$$\partial_\varepsilon  w_{|_{\varepsilon =0}}= \boT\Big(F(\Phi,w,\varepsilon)\Big){|_{\varepsilon =0}}.$$
Here $\boT$ is the resolvent of the linear operator $\boL$.
On the other hand, from \eqref{def2:F} we have
\begin{align*}
F(\Phi,w,\varepsilon){|_{\varepsilon =0}} = &\frac{p(p-1)}{2} \Phi^{p-2}\psi^2 + \frac{(p-1)}{2} (\log r) \Phi^{p-1}\psi + \frac{(p-1)^2}{4} (\log r)^2 \Phi^{p} \\
&+  \left( \frac{d}{2} -1 \right) \frac{\psi}{r^2}  +\frac{q(q-1)}{2} \Phi^{q-2}\psi^2 + \frac{(q-1)}{2} (\log r) \Phi^{q-1}\psi\\
& + \frac{(q-1)^2}{4} (\log r)^2 \Phi^{q}  .
\end{align*}
Thus, from \eqref{behav:phi:origin}, \eqref{behv:infty:Phi}, Lemmas \ref{lem:ker-L} and \ref{lem:cont-boT}, we obtain \eqref{estim-deriv-w-ep}.

This concludes the proof of this proposition.
\end{proof}
As a consequence, we can provide the behavior of this exterior solution at $r_*$.
\begin{corollary}
Let $0<r_*<1$ be a small enough universal constant and $\varepsilon:=\lambda-\lambda_*$ be an arbitrary parameter satisfying \eqref{est:r0-eps}. The solution of \eqref{ODE-pq} given by \eqref{decomp:ext} has the following behavior
\begin{align}
\label{behav:u-ext-r*}
u(r_*)= &A(p,d) r_*^{-\frac{2}{p-1}}  \left(1+ O(\varepsilon |\log r_*| )+ O(\varepsilon^2 |\log r_*|^2 ) +O\left(r_*^{\frac{2(p-q)}{p-1}}\right)\right) \\
&+ \varepsilon \frac{\tilde{c}_1\sin\left(\omega \log(r_*)+\tilde{c}_2\right)}{r_*^{\frac{d-2}{2}}}\big(1+O\left( |\log(r_*)|  r_*^{\sigma -1}\right) +O\left(  r_*^{\sigma -1}\right) \nonumber\\
&\qquad + O(\varepsilon |\log r_*| )  +O(\varepsilon r_*^{1-\sigma} ) +O(\varepsilon r_*^{\sigma-1}|\log r_*|^2 ) \big). \nonumber
\end{align}
\end{corollary}

\begin{proof}
The proof follows in straightforward. We recall from \eqref{decomp:ext} that $$u=r^{\frac{\varepsilon}{2}} \left( \Phi+\varepsilon(\psi+ w)\right). $$ \eqref{behav:phi:origin} gives
$$\Phi(r_*)= A(p,d) r_*^{-\frac{2}{p-1}}  \left(1+O\left(r_*^{\frac{2(p-q)}{p-1}}\right)\right) $$
and \eqref{behv:origin-psi} gives 
$$\psi(r_*)= \frac{\tilde{c}_1\sin\left(\omega \log(r_*)+\tilde{c}_2\right)}{r_*^{\frac{d-2}{2}}}\left(1+O\left( |\log(r_*)|  r_*^{\sigma -1}\right) +O\left(  r_*^{\sigma -1}\right)\right).$$
Moreover, from \eqref{bound:norm-w}, we write
$$ w = r^{-\frac{d-2}{2}} O\left( \varepsilon  r_*^{1-\sigma}\right) $$
In addition, we use \eqref{est:r0-eps} in order to write
$$r_*^{\frac{\varepsilon}{2}} = 1 + O(\varepsilon |\log r_*| ) + O(\varepsilon^2 |\log r_*|^2 ) .$$
Combining this with the previous behaviors, we conclude this proof.
%\begin{align*}
%u_1(r_*)= &\left(1+ O(\varepsilon |\log r_*| )\right) \Bigg( r^{-\frac{2}{p-1}} (1+O(r^2)) \\
%&+ \varepsilon \frac{\tilde{c}_1\sin\left(\omega \log(r)+\tilde{c}_2\right)}{r^{\frac{d-2}{2}}}\left(1+O\left( |\log(r)|  r^{\sigma -1}\right)+O\left(  r^{\sigma -1}\right)+O(\varepsilon r_*^{1-\sigma}|\log r_*|^2 ) \right),
%\end{align*}
\end{proof}
\section{Interior solution}

In this section, we construct the set of interior solutions to \eqref{ODE-pq} in $[0, r_*]$. This family bifurcates from the solution of the Emden-Fowler equation i.e. the positive radially symmetric solution to the following stationnary equation
\begin{align}
&\Delta \zeta + \zeta^p=0 \quad {\rm in} \quad r>0 \label{Emden-Fowler}\\
&\zeta(0)=\theta>0 \quad  {\rm and} \quad \zeta'(0)=0.\nonumber
\end{align}

We denote by $Q$ the solution to \eqref{Emden-Fowler} with initial condition $Q(0)=1$. We first recall the following result that describes its behavior at infinity. For the proof, we refer to \cite{budnor, Di, Jo, YiLi} and the references therein.

\begin{lemma}
There is an $R$ sufficiently large such that, for any $r\geq R$ and $n\in \N$ we have
\begin{equation}
\label{behav:Q-infty}
\left[Q(r)- \frac{A_p}{r^{\frac{2}{p-1}}}\right]^{(n)} = \left[\frac{C\sin\left(\omega\log(r)+D\right)}{r^{\frac{d-2}{2}}}\right]^{(n)} + O\left(r^{-n+2-\sigma-\frac{d}{2}} \right),
\end{equation}
where $C\neq0$ and $D\in \R$.
\end{lemma}

For the linear level, we recall the result in \cite{CRS} about the continuity of the resolvent of the linearized operator
$$H:=-\Delta  -p Q^{p-1}$$
in a neighbourhood of  $Q$ in suitable weighted space. Let  $\tau_1\gg1$ and $Y_{\tau_1}$ is the Banach space of functions on $(0,\tau_1)$ embedded with the following norm
$$\|T\|_{Y_{\tau_1}}=\sup_{0\leq \tau\leq \tau_1}\left\lbrace(1+\tau)^{\frac{2}{p-1}-2}(|T|+\tau|\partial_\tau T|)\right\rbrace.$$
\begin{lemma}[Interior resolvent of $H$ \cite{CRS}]
\label{lem:homo-H}
\noindent{\em 1. Basis of fundamental solutions}: we have
$$H(\Lambda Q)=0, \ \ H\rho=0,$$
where $\Lambda:= \frac{2}{p-1} + \tau \partial_\tau $, with the following asymptotic behavior as $\tau\to+\infty$
\begin{equation}
\label{behav-sol-H-1}
\Lambda Q(\tau)= \frac{1}{\tau^{\frac{d-2}{2}}	} \left( c_5\sin\left(\omega\log(\tau)+c_6\right) +O\left(\frac{1}{\tau^{\sigma-1}}\right)\right)
\end{equation}
and
\begin{equation}
\label{behav-sol-H-2}
\rho(\tau)=\frac{1}{\tau^{\frac{d-2}{2}}	} \left( c_7\sin\left(\omega\log(\tau)+c_8\right) +O\left(\frac{1}{\tau^{\sigma-1}}\right)\right) ,
\end{equation}
where $c_5, c_7\neq 0$, $c_6, c_{8}\in\mathbb{R}$.\\
\noindent{\em 2. Continuity of the resolvent}: let the inverse
$$\boS(f) = \left(\int_0^\tau f\rho  \ {s}^{d-1}ds\right)\Lambda Q- \left(\int_0^\tau f\Lambda Q  \ {s}^{d-1}ds\right)\rho$$
then
\begin{equation}
\label{cont-resolvent-S}
\|\boS(f)\|_{  Y_{\tau_1}}\lesssim \sup_{0\leq \tau\leq \tau_1}\left\{(1+\tau)^{\frac{2}{p-2}}|f|\right\}.
\end{equation}
\end{lemma}

\begin{proof}
The proof is already given in \cite{CRS}. For the sake of completness, we provide the proof of the continuity of the resolvent. We recall that near the origin we have the following bound of $\rho$
\begin{equation}
\label{estrhorigin}
|\rho(\tau)|\lesssim \frac{1}{\tau^{d-2}} , \ |\partial_\tau \rho(\tau)|\lesssim \frac{1}{\tau^{d-1}}  \ \mbox{as}\ \ \tau\to 0.
\end{equation}
Thus, from the definition of $\boS(f)$ and the estimate \eqref{estrhorigin}, we get
\begin{align*}
&|\boS(f)|= \left|\left(\int_0^\tau f\rho {s}^{d-1}ds\right)\Lambda Q- \left(\int_0^\tau f\Lambda Q {s}^{d-1}ds\right)\rho\right|\\
&\lesssim \left(\int_0^\tau sds+\frac{1}{\tau^{d-2}} \int_0^\tau {s}^{d-1}ds \right)\sup_{0\leq \tau\leq 1}|f| \lesssim \sup_{0\leq \tau\leq \tau_1} \left\lbrace (1+\tau)^{\frac{2}{p-1}}|f| \right\rbrace ,
\end{align*}
and
\begin{align*}
&|\tau\partial_\tau \boS(f)|= \left|\left(\int_0^\tau f\rho {s}^{d-1}ds\right)\tau\partial_\tau\Lambda Q- \left(\int_0^\tau f\Lambda Q {s}^{d-1}ds\right)\tau\partial_\tau\rho\right|\\
&\lesssim \left(\tau^{2}\int_0^\tau sds+\frac{1}{\tau^{d-2}} \int_0^\tau {s}^{d-1}ds \right)\sup_{0\leq \tau\leq 1}|f| \lesssim \sup_{0\leq \tau\leq \tau_1} \left\lbrace (1+r)^{\frac{2}{p-1}}|f| \right\rbrace ,
\end{align*}
for $0\leq \tau\leq 1$. Similarly, for $1\leq r \leq r_1$ we infer from \eqref{behav-sol-H-1} and \eqref{behav-sol-H-2} that
\begin{align*}
&(1+\tau)^{\frac{2}{p-1}-2}|\boS(f)|= (1+\tau)^{\frac{2}{p-1}-2}\left|\left(\int_0^\tau f\rho {s}^{d-1}ds\right)\Lambda Q- \left(\int_0^\tau f\Lambda Q {s}^{d-1}ds\right)\rho\right|\\
&\lesssim (1+\tau)^{-\sigma-1}\left(\int_0^\tau |f| (1+s)^{\frac{d}{2}} ds\right)\\
&\lesssim (1+\tau)^{-\sigma-1}\left(\int_0^\tau (1+s)^{\sigma} ds\right)\sup_{0\leq \tau\leq \tau} \left\lbrace (1+\tau)^{\frac{2}{p-1}}|f| \right\rbrace \\
&\lesssim \sup_{0\leq \tau\leq \tau_1} \left\lbrace (1+\tau)^{\frac{2}{p-1}}|f| \right\rbrace
\end{align*}
and
\begin{align*}
&(1+\tau)^{\frac{2}{p-1}-2}|\tau\partial_\tau \boS(f)|\\
&= (1+\tau)^{-\frac{2}{p-1}-2}\left|\left(\int_0^\tau f\rho {s}^{d-1}ds\right)\tau\partial_\tau \Lambda Q- \left(\int_0^\tau f\Lambda Q {s}^{d-1}ds\right)\tau\partial_\tau \rho\right|\\
&\lesssim (1+\tau)^{-\sigma-1}\left(\int_0^\tau |f| (1+s)^{\frac{d}{2}} ds\right)\\
&\lesssim (1+\tau)^{-\sigma-1}\left(\int_0^\tau (1+s)^{\sigma} ds\right)\sup_{0\leq \tau\leq \tau_1} \left\lbrace (1+\tau)^{\frac{2}{p-1}}|f| \right\rbrace \\
&\lesssim \sup_{0\leq \tau\leq \tau_1} \left\lbrace (1+\tau)^{\frac{2}{p-1}}|f| \right\rbrace.
\end{align*}
This ends the proof of \eqref{cont-resolvent-S}, the continuity estimate of the resolvent $\boS$.
\end{proof}

As a consequence, we are able to construct the set of interior solutions.

\begin{proposition}[Construction of the interior solution]
\label{prop:int-sol}
Let $r_*>0$ be small enough and let $0<\lambda<d$. Then, for any $\theta>r_*^{-\frac{2}{p-1}}$ there exists a solution $u$ to \eqref{ODE-pq} on $0\leq r\leq r_*$,
of the form
\begin{equation}
\label{decomp:int}
u=\theta(Q+\theta^{1-p} T)\left( \theta^{\frac{p-1}{2}} r\right)
\end{equation}
with
\begin{equation}
\label{estim-inter-error}
\|T\|_{Y_{\tau_1}}\lesssim \theta^{q-1} (1+\tau_1)^{-\frac{2(q-1)}{p-1}},
\end{equation}
where $\tau_1:= \theta^{\frac{p-1}{2}} r_*>1$.
\end{proposition}

\begin{remark}
Note that, the initial condition $u(0)=\theta$ imposes that $T(0)=0$.
\end{remark}

\begin{proof}
The proof is exactly the same as Proposition 2.4 in \cite{CRS}. We will sketch it.
Let $u$ be a solution to \eqref{ODE-pq} on $[0,r_*]$ given by
$$u=\theta(Q+\theta^{1-p} T)\left( \theta^{\frac{p-1}{2}} r\right).$$
This means that $T$ is a solution to the following equation
$$H(T)=J[Q, \theta]T \ \text{on} \ 0\leq \tau \leq \tau_1$$
where
$$\tau= \theta^{\frac{p-1}{2}} r $$
and
\begin{align}
J[Q, \theta]T(\tau) = &-(\theta^{1-p}\tau^2-\lambda ) \left( Q(\tau)+\theta^{1-p} T(\tau) \right) + \theta^{q-1} \left( Q(\tau)+\theta^{1-p} T(\tau) \right)^q \label{def:J-T} \\
&+ p(p-1)\theta^{1-p}\left(\int_0^1(1-s)(Q(\tau)+s\theta^{1-p}T(\tau))^{p-2}ds\right)T^2(\tau).\nonumber
\end{align}
We shall prove that $T\mapsto \boS\left(J[Q, \theta]T \right)$ is a contraction mapping on the Banach space
$$E_{\theta,\tau_1} := \left\{ T \in Y_{\tau_1}; \ \|T\|_{Y_{\tau_1}}\lesssim \theta^{q-1} (1+\tau_1)^{-\frac{2(q-1)}{p-1}}  \right\} $$
It is sufficient to show that, if $\|T\|_{Y_{\tau_1}}\lesssim \theta^{q-1} (1+\tau_1)^{-\frac{2(q-1)}{p-1}}$, then
\begin{eqnarray}
\label{est:resol-T-in-Y1}
&&\sup_{0\leq \tau\leq \tau_1}(1+\tau)^{\frac{2}{p-1}}|J[Q, \theta]T| \lesssim \theta^{q-1} (1+\tau_1)^{-\frac{2(q-1)}{p-1}},\\
\label{est:resol-T-in-Y2}
&&\sup_{0\leq \tau\leq \tau_1}(1+\tau)^{\frac{2}{p-1}}|J[Q, \theta]T_1 - J[Q, \theta]T_2| \lesssim   r_*^{\frac{2(p-q)}{p-1}}  \|T_1-T_2\|_{Y_{\tau_1}}.
\end{eqnarray}
Indeed, combining $  r_*^{\frac{2(p-q)}{p-1}} \ll 1$, the resolvent estimate \eqref{cont-resolvent-S} and the nonlinear estimates \eqref{est:resol-T-in-Y1} and \eqref{est:resol-T-in-Y2} with the Banach fixed point theorem, we infer that there exists a unique solution
$$T=\boS(J[Q, \theta]T), $$
which satisfies \eqref{estim-inter-error}.

To prove \eqref{est:resol-T-in-Y1}, we shall first show that $u$ is always positive if $T\in Y_{\tau_1}$. We know that
 $$|T(\tau)| \lesssim \theta^{q-1} (1+\tau)^{2-\frac{2q}{p-1}} \lesssim \theta^{q-1} (1+\tau_1)^{\frac{2(p-q)}{p-1}} (1+\tau)^{-\frac{2}{p-1}}.$$
for any $\tau\in [0,\tau_1]$. Note that, by definition of $\tau_1$, we have
\begin{equation}
\label{theta-tau1}
\theta^{\frac{1-p}{2}}\tau_1=r_*\ll 1.
\end{equation}

Therefore, from the bound on $Q$, we obtain
 $$\theta^{1-p}|T(\tau)| \lesssim \theta^{q-p} \tau_1^{\frac{2(p-q)}{p-1}} (1+\tau)^{-\frac{2}{p-1}}\lesssim r_*^{\frac{2(p-q)}{p-1}} Q(\tau) \ll Q(\tau) .$$
Hence,
$$ \frac{1}{2} Q(\tau)  < Q(\tau) +\theta^{1-p}T(\tau) < 2 Q(\tau) $$
for any $\tau\in [0,\tau_1]$. Since $T \in Y_{r_1}$, we infer from \eqref{def:J-T} and \eqref{theta-tau1} that, for $\tau \in [0,\tau_1]$, we have
\begin{align*}
&\left| J[Q, \theta]T(\tau) \right| \\
&\lesssim (r_*^2+\lambda) Q(\tau)+ \theta^{q-1}  Q(\tau)^q+  \theta^{1-p} Q(\tau)^{p-2} T^2(\tau)\\
&\lesssim (r_*^2+\lambda) (1+\tau)^{-\frac{2}{p-1}} + \theta^{q-1}  (1+\tau)^{-\frac{2q}{p-1}}  +\theta^{1-p} (1+\tau)^{-2+\frac{2}{p-1}} (1+\tau)^{4-\frac{4}{p-1}} \|T\|^2_{Y_{\tau_1}}\\
&\lesssim (r_*^2+\lambda) (1+\tau)^{-\frac{2}{p-1}}+  \theta^{q-1}  (1+\tau)^{-\frac{2q}{p-1}}  +\theta^{1-p} (1+\tau)^{2-\frac{2}{p-1}} (\theta^{q-1} (1+\tau_1)^{-\frac{2(q-1)}{p-1}} )^2 \\
&\lesssim (r_*^2+\lambda) (1+\tau)^{-\frac{2}{p-1}}+  \theta^{q-1}  (1+\tau)^{-\frac{2q}{p-1}}  +r_*^{\frac{2(p-q)}{p-1}} (1+\tau)^{-\frac{2}{p-1}} \theta^{q-1} (1+\tau_1)^{-\frac{2(q-1)}{p-1}}  \\
&\lesssim (1+\tau)^{-\frac{2}{p-1}} + \theta^{q-1}  (1+\tau)^{-\frac{2q}{p-1}}    .
\end{align*}
This concludes \eqref{est:resol-T-in-Y1}. We argue then similarly to show \eqref{est:resol-T-in-Y2}. More precisely, we write

\begin{align*}
&J[Q, \theta]T_1 - J[Q, \theta]T_2\\
&= -(r^2-\lambda ) \theta^{1-p} \left(T_1(\tau)-T_2(\tau) \right) \\
&+ p \theta^{q-p} \int_0^1(Q(\tau)+\theta^{1-p}(sT_1(\tau) + (1-s) T_2(\tau))  )^{q-1} ds(T_1(\tau)-T_2(\tau))  \\
&+ p(p-1)\theta^{1-p}\int_0^1(1-s)(Q(\tau)+s\theta^{1-p}T(\tau))^{p-2}ds(T_1-T_2)(T_1+T_2).
\end{align*}
Thus,
\begin{align*}
|J[Q, \theta]T_1 - J[Q, \theta]T_2| \lesssim &-(r_*^2-\lambda ) \theta^{q-p} \tau_1^{\frac{2(p-q)}{p-1}} (1+\tau)^{-\frac{2}{p-1}} \|T_1-T_2\|_{Y_{\tau_1}}  \\
&+ p \theta^{q-p} Q^{q-1}(\tau) \tau_1^{\frac{2(p-q)}{p-1}} (1+\tau)^{-\frac{2}{p-1}} \|T_1(\tau)-T_2\|_{Y_{\tau_1}}   \\
&+ p(p-1)\theta^{q-p} Q^{p-2}(\tau)  \tau_1^{\frac{2(p-q)}{p-1}} (1+\tau)^{-\frac{2}{p-1}}  \|T_1-T_2\|_{Y_{\tau_1}}  \\
 \lesssim &  r_*^{\frac{2(p-q)}{p-1}} (1+\tau)^{-\frac{2}{p-1}}  \|T_1-T_2\|_{Y_{\tau_1}}.
\end{align*}

 This finishes the proof of this lemma.
\end{proof}
As a consequence, the behavior of this interior solution at $r_*$ is given by the following corollary.
\begin{corollary}
Let $0<r_*<1$ be a small enough universal constant. For any $\theta>r_*^{-\frac{2}{p-1}}$, the solution of \eqref{ODE-pq} given by \eqref{decomp:int} has the following behavior
\begin{align}
\label{behav:u-int-r*}
u(r_*)= &A_p r_*^{-\frac{2}{p-1}}  \left(1+ O\left(r_*^{\frac{2(p-q)}{p-1}} \right)\right)  +  O\left(\theta(\theta^{\frac{p-1}{2}} r_*)^{2-\sigma-\frac{d}{2}} \right)  \nonumber\\
 & + \frac{C \theta^{\frac{(p-1)(1-\sigma)}{2}} \sin\left(\omega(\log(r_*)+\frac{2}{p-1} \log(\theta)) +D\right)}{r_*^{\frac{d-2}{2}}}  \left(1+ O\left(r_*^{\frac{2(p-q)}{p-1}} \right)\right)  .
\end{align}
\end{corollary}
The proof is a direct consequence of \eqref{behav:Q-infty} and Proposition \eqref{prop:int-sol}.

\section{Matching}
In this section we will be gluing the interior solution with the exterior one at $r_*$. For that, we shall compare the behaviors of both solutions at $r_*$. The matching will hold then using the Brouwer fixed point theorem on the parameters $\varepsilon$ and $\theta$. We show that there exist a countable number of smooth solutions to \eqref{ODE-pq} that converges to the singular solution $\Phi$ when the initial condition goes to infinity.

\begin{proof}[Proof of Theorem \ref{thm:matching}]
We denote by $u_1$ the exterior solution to \eqref{ODE-pq} given by Proposition \ref{prop:ext-sol} and  $u_0$ the interior solution to \eqref{ODE-pq} given by Proposition \ref{prop:int-sol}. The matching of both solutions at $r=r_*$ remains to establish
\begin{equation}
\label{matching-system}
\left\lbrace \begin{array}{c}
u_1(r_*)-u_0(r_*)=0\\
u'_1(r_*)-u'_0(r_*)=0.
\end{array}
\right.
\end{equation}
For that, we introduce the function
$$F(\varepsilon,\theta) =\left( \begin{array}{c}
F_1(\varepsilon,\theta)\\
F_2(\varepsilon,\theta)
\end{array}
\right) :=\left( \begin{array}{c}
r_*^{\frac{d-2}{2}} \left( u_1(r_*)-u_0(r_*) \right) \\
r_*^{\frac{d}{2}}\left( u'_1(r_*)-u'_0(r_*) \right)
\end{array}
\right). $$
From \eqref{behav:u-ext-r*} and \eqref{behav:u-int-r*}, we have
\begin{align}
 F_1(\varepsilon,\theta)= &   \varepsilon \tilde{c}_1\sin\left(\omega \log(r_*)+\tilde{c}_2\right)- C \theta^{\frac{1-\sigma}{\alpha}}\sin\left(\omega(\log(r_*)+\alpha \log(\theta)) +D\right) \nonumber\\
& +O(\varepsilon r_*^{\sigma-1} |\log r_*| )+O\left( \varepsilon  r_*^{\sigma -1}\right)+ O(\varepsilon^2 r_*^{\sigma-1} |\log r_*|^2 ) \nonumber \\
& + O(\varepsilon^2 |\log r_*| ) + O(\varepsilon^2 r_*^{1-\sigma} |\log r_*|^2 )+  O\left(\theta^{\frac{2(1-\sigma)}{\alpha}}  r_*^{1-\sigma} \right) \nonumber\\
&  + O\left( r_*^{\sigma -1} r_*^{\frac{2(p-q)}{p-1}} \right) + O\left( \theta^{\frac{1-\sigma}{\alpha}} r_*^{\frac{2(p-q)}{p-1}} \right) \label{behavior:F1}
\end{align}
and
\begin{align}
F_2(\varepsilon,\theta)= &   \varepsilon \tilde{c}_1\omega  \cos\left(\omega \log(r_*)+\tilde{c}_2\right)- \varepsilon \tilde{c}_1 \frac{(d-2)}{2}  \sin\left(\omega \log(r_*)+\tilde{c}_2\right) \nonumber \\
&- C \omega \theta^{\frac{1-\sigma}{\alpha}}  \cos\left(\omega(\log(r_*)+\alpha \log(\theta)) +D\right)\nonumber \\
&+ C \frac{(d-2)}{2}\theta^{\frac{1-\sigma}{\alpha}}\sin\left(\omega(\log(r_*)+\alpha \log(\theta)) +D\right) \nonumber \\
& +O(\varepsilon r_*^{\sigma-1} |\log r_*| )+O\left( \varepsilon  r_*^{\sigma -1}\right)+ O(\varepsilon^2 r_*^{\sigma-1} |\log r_*|^2 ) \nonumber \\
& + O(\varepsilon^2 |\log r_*| ) + O(\varepsilon^2 r_*^{1-\sigma} |\log r_*|^2 )+  O\left(\theta^{(p-1)(1-\sigma)}  r_*^{1-\sigma} \right) \nonumber \\
& + O\left( r_*^{\sigma -1} r_*^{\frac{2(p-q)}{p-1}} \right) + O\left( \theta^{\frac{1-\sigma}{\alpha}} r_*^{\frac{2(p-q)}{p-1}} \right)  , \label{behavior:F2}
\end{align}
where $\alpha:=\frac{2}{p-1}$. We write
\begin{align*}
\sin\left(\omega(\log(r_*)+\alpha \log(\theta)) +D\right)&= \sin\left(\omega\log(r_*) +\tilde{c}_2\right)\cos\left(\omega \alpha  \log(\theta) +D-\tilde{c}_2\right)\\
& + \cos\left(\omega\log(r_*) +\tilde{c}_2\right)\sin\left(\omega \alpha \log(\theta) +D-\tilde{c}_2\right)
\end{align*}
and
\begin{align*}
\cos\left(\omega(\log(r_*)+\alpha \log(\theta)) +D\right)&= \cos\left(\omega\log(r_*) +\tilde{c}_2\right)\cos\left(\omega \alpha  \log(\theta) +D-\tilde{c}_2\right)\\
& - \sin\left(\omega\log(r_*) +\tilde{c}_2\right)\sin\left(\omega \alpha \log(\theta) +D-\tilde{c}_2\right).
\end{align*}
Next, we choose $\varepsilon_*,\theta_*$ and $r_*$ that satisfy
\begin{equation}
\label{relation-1}
\left\lbrace \begin{array}{l}
\varepsilon_*= \frac{C}{\tilde{c}_1}\theta_*^{\frac{1-\sigma}{\alpha}} \cos\left(\frac{2\omega}{p-1} \log(\theta_*) +D-\tilde{c}_2\right)\\
\frac{2\omega}{p-1} \log(\theta_*) +D-\tilde{c}_2= 2 n \pi \quad {\rm for} \quad n\in \N\\
r_*= O( \theta_*^{\frac{1-\sigma}{\alpha \sigma}} )
\end{array}
\right.
\end{equation}
Hence, we get
\begin{align*}
F_1(\varepsilon_*,\theta_*)=&  O\left( r_*^{\sigma -1} r_*^{\frac{2(p-q)}{p-1}} \right) + O\left( \theta^{\frac{1-\sigma}{\alpha}} r_*^{\frac{2(p-q)}{p-1}} \right)  + O\left(\varepsilon_*^2 r_*^{1-\sigma} |\log r_*|^2 \right)\\
=&   O\left( \varepsilon _*^{1- \frac{1}{\sigma} } \varepsilon _*^{\frac{\alpha}{\sigma}(p-q)} \right) + O\left( \theta_*^{\frac{1-\sigma}{\alpha}} r_*^{ \alpha (p-q)} \right) + O\left(\varepsilon_*^2 r_*^{1-\sigma} |\log r_*|^2 \right)\\
=&  O\left( \varepsilon _* \varepsilon _*^{\frac{\alpha}{\sigma}(\beta-q)} \right) + O\left( \theta_*^{\frac{1-\sigma}{\alpha}} r_*^{ \alpha (p-q)} \right)  + O\left(\varepsilon_*^2 r_*^{1-\sigma} |\log r_*|^2 \right)
\end{align*}
 and $$F_2(\varepsilon_*,\theta_*)=  O\left( \varepsilon _* \varepsilon _*^{\frac{\alpha}{\sigma}(\beta-q)} \right) + O\left( \theta_*^{\frac{1-\sigma}{\alpha}} r_*^{ \alpha (p-q)} \right)  + O\left(\varepsilon_*^2 r_*^{1-\sigma} |\log r_*|^2 \right) ,$$
 where $\beta = \frac{p+1}{2}.$ Note that, $\beta>q$ by definition of $q$. Let us prove that there exists $(\varepsilon,\theta)$ in a neighborhood of $(\varepsilon_*,\theta_*)$ which satisfy $F(\varepsilon,\theta)=0.$ We start first by calculating its Jacobian which is given by
 \begin{align*}
\frac{\partial F\left(\varepsilon_*, \theta_*\right)}{\partial\left(\varepsilon, \theta\right)} =&\left[\begin{array}{ll}
\tilde{c}_1\sin \tau_*, & - C \theta_*^{\frac{1-\sigma}{\alpha}-1} \gamma_2 \\
\tilde{c}_1\gamma_3 , & C \theta^{\frac{1-\sigma}{\alpha}-1}   \left(  \omega  \gamma_1 +    \frac{(d-2)}{2} \gamma_2  \right)
\end{array}\right] + O\left(\varepsilon_*^{\frac{1}{\sigma}} |\log \varepsilon_*|^2 \right)\\
&+O\left( \theta_*^{\frac{1-\sigma}{\alpha}-1} r_*^{ \alpha (p-q)} \right)
 \end{align*}
where $\tau_* =\omega \log(r_*)+\tilde{c}_2$, $\gamma_1:= \omega \sin \tau_* -  \frac{1-\sigma}{\alpha} \cos\tau_* $,  $\gamma_2:= \frac{1-\sigma}{\alpha} \sin \tau_* + \omega \cos \tau_* $ and $\gamma_3= \omega  \cos \tau_* -   \frac{(d-2)}{2}  \sin\tau_*$. Next, we set
$$G(x,y):= F(\varepsilon_*+x,\theta_*+ \theta_*^{\frac{\sigma-1}{\alpha}+1} y). $$
From \eqref{behavior:F1} and \eqref{behavior:F1}, we have
\begin{align*}
G(x,y)= \boK + L \left(\begin{array}{ll} x \\ y \end{array}\right) + E(x,y)
\end{align*}
with $\boK= O\left( \varepsilon _* \varepsilon _*^{\frac{\alpha}{\sigma}(\beta-q)} \right) $ is a constant vector independent of $(x,y)$, $L$ is an invertible matrix given by
$$L= \left[\begin{array}{ll}
\tilde{c}_1\sin \tau_*, & - C \gamma_2 \\
\tilde{c}_1\gamma_3 , & C  \left(  \omega  \gamma_1 +    \frac{(d-2)}{2} \gamma_2  \right)
\end{array}\right] $$
and $E(x,y)$ contains all the remanding terms. Our goal is to prove that there exists a root for the equation $G(x,y)=0$. This is equivalent to consider the fixed point problem $\left(\begin{array}{ll} x \\ y \end{array}\right) = J(x,y) $ where $J(x,y):= - L^{-1} \boK - L^{-1} E(x,y) .$ Note that $J$ maps the set $B$ into itself, where
$$B:=\left\lbrace (x,y):  \ |(x,y)|_B := (x^2+y^2)^{\frac{1}{2}} \leq M \theta_*^{\frac{1-\sigma}{\alpha}}  \varepsilon _*^{\frac{\alpha}{\sigma}(\beta-q)} \right\rbrace  $$
and $M>0$ does not depend on $\varepsilon_*$ and $\theta_*$. Indeed, from \eqref{behavior:F1} and \eqref{behavior:F2}, we have
\begin{align*}
|J(x,y)|_B \lesssim \theta_*^{\frac{1-\sigma}{\alpha}}  \varepsilon _*^{\frac{\alpha}{\sigma}(\beta-q)} + |x| \varepsilon_*^{\frac{1}{\sigma}} |\log \varepsilon_*|^2 + |y|  \varepsilon _*^{\frac{\alpha}{\sigma}(\beta-q)} .
\end{align*}
Hence, using the Brouwer fixed theorem, $J$ has a fixed point in B. Thus, we proved that there exists $(\varepsilon,\theta)$ solution to $F(\varepsilon,\theta)=0$ with
$$\varepsilon = \varepsilon_* \left(1+ O\left( \varepsilon _*^{\frac{\alpha}{\sigma}(\beta-q)}  \right) \right) $$
and
$$\theta = \theta_* \left(1+ O\left( \theta _*^{\frac{1-\sigma}{\sigma}(\beta-q)}  \right) \right). $$
We include this into \eqref{relation-1}, we obtain
\begin{align*}
 \varepsilon + \varepsilon_*  O\left( \varepsilon _*^{\frac{\alpha}{\sigma}(\beta-q)}  \right) &= \frac{C}{\tilde{c}_1}\theta_*^{\frac{1-\sigma}{\alpha}} \cos\left( \alpha \omega \log\left(\theta \right) - \alpha \omega \log\left( \frac{\theta}{\theta_*} \right) + D-\tilde{c}_2\right)  \\
 &= \frac{C}{\tilde{c}_1}\theta_*^{\frac{1-\sigma}{\alpha}} \cos\left( \alpha \omega \log\left(\theta \right) + O\left( \theta _*^{\frac{1-\sigma}{\sigma}(\beta-q)}  \right)  + D-\tilde{c}_2\right) .
\end{align*}
Thus, using \eqref{relation-1} again, we get
$$ \varepsilon =  \frac{C}{\tilde{c}_1}\theta_*^{\frac{1-\sigma}{\alpha}} \cos\left( \alpha \omega \log\left(\theta \right)  + D-\tilde{c}_2\right) \left( 1+ O\left( \theta _*^{\frac{1-\sigma}{\sigma}(\beta-q)}  \right) \right)   $$

Note that, there exists only a countable number of $(\varepsilon,\theta)$. This is why we are going to note them $(\varepsilon_n,\theta_n)$. \eqref{Lambda-theta} is obtained by taking $\tilde{\theta}_n:= \theta_*$, $C_1=\frac{C}{\tilde{c}_1}$, $C_2=  D-\tilde{c}_2 +\frac{\pi}{2}$ and by using the fact that $\varepsilon_n= \lambda_n - \lambda_*$.

Note that the convergence of $u$ to the singular solution $\Phi$ holds by definition of $u$ in \eqref{decomp:ext} and \eqref{decomp:int}.
\end{proof}
\appendix

\section{The singular solution}

\subsection{Proof of Theorem \ref{Thm-sing}}

The proof is very similar to the one of Theorem 1.1 in \cite{MePe} and Theorem 1.1 in \cite{HSKW}. The existence and the role of $\lambda_*$ will not be considered here as it follows exactly like in \cite{HSKW}. We refer to it for more details. However, it is not obvious to derive the influence of the term $|u|^qu$ on the asymptotic behavior near the origin.  We are going to write only the main difficulties in our case for the sake of completeness.

\begin{proof}
We set
$$v(r)=A^{-1} r^{2 /(p-1)} u(r),$$
where $u$ is a solution to \eqref{ODE-pq}. It satisfies

$$
\left\{\begin{array}{ll}
v^{\prime \prime}+\frac{K-1}{r} v^{\prime}+\frac{A^{p-1}}{r^{2}}\left(v^{p}-v\right)+(\lambda-r^2) v+A^{q-1}r^{-\frac{2(q-1)}{p-1}}v^{q}=0, & 0<r<1 \\
v>0, & 0<r<1 \\
v(r) \rightarrow 1 \quad \text { as } \quad r \rightarrow 0
\end{array}\right.
$$
where
$$K=d-\frac{4}{p-1}, \quad \alpha=\frac{2}{p-1}.$$

We introduce
$$
t=\frac{1}{m} \log r, \quad y(t)=v(r)
$$
This yields for $y:$
$$
\text { (II) }\left\{\begin{array}{ll}
y^{\prime \prime}+\mu y^{\prime}-y+y^{p}+ \beta e^{2 m t} y - m^2 e^{4 m t} ) y  + A^{q-1} m^2 e^{\frac{2(p-q)}{p-1}mt}  y^{q}=0,  -\infty<t<0 \\
y>0, \qquad -\infty<t<0 \\
y(0)=0, \quad  y(t) \rightarrow 1 \quad \text { as } \quad t \rightarrow-\infty
\end{array}\right.
$$
where
$$
m=\{\alpha(d-2-\alpha)\}^{-\frac{1}{2}}
$$
and
$$
\mu=m(d-2-2 \alpha), \quad \beta=\frac{\lambda}{\alpha(d-2-\alpha)}.
$$
Note that, $d-2-\alpha=\frac{d}{2}+ \sigma -2 >0 since \frac{d}{2} -2 \geq -\frac{1}{2}$ when $d\geq 3$ and $\sigma \geq 1 $.
We shift
$$t=\tau-\frac{\log \beta}{2 m}.$$
Yields,
$$\left\{\begin{array}{l}
y^{\prime \prime}+\mu y^{\prime}-y+y^{p}+e^{2 m \tau} y- \frac{m^2}{\beta^2} e^{4 m \tau}  y  + A^{q-1} m^2\beta^\frac{q-p}{p-1}  e^{\frac{2(p-q)}{p-1}m\tau}  y^{q}=0, \-\infty<\tau<T  \\
y>0, \qquad  -\infty<\tau<T \\
y(T)=0, \quad y(\tau) \rightarrow 1 \quad \text { as } \tau \rightarrow-\infty
\end{array}\right. $$
where
$$T=\frac{\log \beta}{2 m}$$
We set
$s=-\tau \quad$ and $\quad \eta(s)=y(\tau)-1.$
Then $\eta$ satisfies
\begin{equation}
\label{eq:eta}
\eta^{\prime \prime}-\alpha \eta^{\prime}+(p-1) \eta=f(s), \quad-T<s<\infty \quad \eta(s) \rightarrow 0 \text { as } s \rightarrow \infty
\end{equation}
where
\begin{equation}
\label{def-f}
f(s)=-e^{-2 m s}\{1+\eta(s)\}-\varphi_1(\eta)-\varphi_2(\eta),
\end{equation}
with
$$
\varphi_1(\eta)=(1+\eta)^{p}-1-p \eta
$$
and $$ \varphi_2(\eta)= -\frac{m^2}{\beta^2} e^{-4 m s}  \{1+\eta(s)\}  + A^{q-1} m^2\beta^\frac{q-p}{p-1}  e^{-\frac{2(p-q)}{p-1}ms}  \{1+\eta(s)\}^{q}.$$

 We distinguish the following three cases:
$$\text { (Case 1) } p-1>\frac{\mu^{2}}{4}, \quad(\text { Case } 2) \  p-1=\frac{\mu^{2}}{4}, \quad(\text { Case } 3) \ p-1<\frac{\mu^{2}}{4} .$$
Let us investigate (Case 1). We set
$$\nu=\sqrt{p-1-\frac{\alpha^{2}}{4}}$$
From the method of variation of parameters, the solution to \eqref{eq:eta} can be written as
\begin{equation}
\label{eq:eta-int}
\eta(s)=\frac{1}{\nu} e^{\frac{\alpha}{2} s} \int_{s}^{\infty} e^{-\frac{\mu}{2} \sigma} \sin (\nu(\sigma-s)) f(\sigma, \eta) d \sigma
\end{equation}
Let $S_*>0$. It is sufficient to implement a Banach fixed point argument on the space $\Xi_{S_*}$ of functions on $[S_*,\infty)$ endowed with the norm
$$ \|\eta\|_{\Xi_{S_*}} := \sup_{s \geq  S_* }  \left\lbrace \left( |\eta| + |\eta '|  \right)  e^{\frac{2(p-q)}{p-1}ms} \right\rbrace$$
for
$$ \eta=\boJ[\eta],$$
with $$ \boJ[\eta] = \frac{1}{\nu} e^{\frac{\mu}{2} s} \int_{s}^{\infty} e^{-\frac{\mu}{2} \sigma} \sin (\nu(\sigma-s)) f(\sigma, \eta) d \sigma.$$
More precisely, by the definition of $f$ in \eqref{def-f}, we have
$$|\boJ[\eta]| \lesssim  e^{-\frac{2(p-q)}{p-1}ms}.$$
In addition, differentiating \eqref{eq:eta-int}, we get
$$ \left| \frac{d}{ds} \boJ[\eta] \right| \lesssim |f(s,\eta)| +  e^{\frac{\mu}{2} s} \int_{s}^{\infty} e^{-\frac{\nu}{2} \sigma} |\sin (\nu(\sigma-s))| f(\sigma, \eta) d \sigma \lesssim  e^{-\frac{2(p-q)}{p-1}ms} .$$
Similarly, we have
$$  \|\boJ[\eta_1]-\boJ[\eta_2]\|_{\Xi_{S_*}} \lesssim e^{-\frac{2(p-q)}{p-1}m S_*} \|\eta_1-\eta_2\|_{\Xi_{S_*}}.$$

This shows the local existence of the singular solution. Note that the other cases can be proved similarly. The global existence can be obtained exactly as in \cite{HSKW} using a shooting argument. Since the proof is similar, we omit it.

Let us show \eqref{behav:deriv-phi:origin}. From the above calculation, we have

$$ \Phi(r)=A r^{-\frac{2}{p-1}} y\left(\frac{\log(r)}{m} \right) = A r^{-\frac{2}{p-1}} \left[ 1 + \eta\left(-\frac{\log(r)}{m} \right) \right].$$
Hence,
$$ \Phi'(r)=  -\frac{2A}{p-1} r^{-\frac{2}{p-1}-1} \left[ 1 +\eta\left(-\frac{\log(r)}{m} \right) + \frac{p-1}{2m} \eta'\left(-\frac{\log(r)}{m} \right) \right].$$
Combining this with $\|\eta\|_{\Xi_{S_*}}  \lesssim 1$, we conclude \eqref{behav:deriv-phi:origin}. This finishes the proof of    this theorem.

%The remaining part of the proof is using the fixed point theorem and it follows exactly as in \cite{MePe} (see the proof of Lemma 3.1).
\end{proof}

\subsection{The behavior at infinity}

 The idea of the proof of the following lemma was first used in \cite[Theorem 2]{BPT} for \eqref{ODE-p}. We provide a complete proof in our setting.
\begin{lemma}
\label{Lem:behav-Phi-infrty}
Let $p>1$ and $q>1$. There exists $K>0$ such that
\begin{equation}
\label{behv:infty:Phi}
\Phi= K \exp\left(-\frac{r^2}{2}\right) r^{\frac{\lambda-d}{2}} \left(1+O\left(\frac{1}{r^2}\right)\right), \textrm{ as } r\to+\infty
\end{equation}
\end{lemma}
\begin{proof}
First, we consider the change of variable
\begin{equation}
\label{change1-Phi}
\Phi= e^{\frac{r^2}{2}} \Psi.
\end{equation}
Thus, the equation for $\Psi$ is given by
\begin{equation}
\label{ODE-PsiPhi}
\Psi''+\left(\frac{d-1}{r}+2r \right)\Psi' +(\lambda+d ) \Psi +\exp\left(\frac{(p-1)}{2}r^2\right)\Psi^{p}+\exp\left(\frac{(q-1)}{2}r^2\right)\Psi^{q}=0, \quad r>0,
\end{equation}
Since $\lim_{r\to \infty}\Phi(r)=0$, we infer from \eqref{change1-Phi} that
\begin{equation}
\label{exp-bound-Psi}
\Psi(r)\lesssim e^{-\frac{r^2}{2}} \quad {\rm for \ any} \  r>1.
\end{equation}
As a consequence, we have
%\begin{claim}
\begin{equation}
\label{limZ-infty}
\lim_{r\to \infty} \frac{\Psi'(r)}{\Psi(r)} \in [-\infty,0] \ {\rm exists}.
\end{equation}
%\end{claim}
%\begin{proof}
Indeed, we shall first prove that $\Psi$ is decreasing for $r>0$ i.e. $\Psi'<0$. We rewrite \eqref{ODE-PsiPhi} as follows
\begin{equation}
\label{ODE-PsiPhi1}
\left(\Psi' r^{d-1} e^{r^2} \right) ' =-(\lambda+d ) r^{d-1}e^{r^2}\Psi -r^{d-1} \exp\left(\frac{(p+1)r^2}{2}\right)\Psi^{p}-\exp\left(\frac{(q+1)}{2}r^2\right)\Psi^{q},
\end{equation}
when $r>0.$ From \eqref{change1-Phi}, \eqref{behav:phi:origin} and the fact that $\sigma:=\frac{d}{2}-\frac{2}{p-1}>1$, we have
$$\lim_{r\to 0}  \Psi'(r) r^{d-1} e^{r^2} =A(p,d) \lim_{r\to 0}  r^{-\frac{2}{p-1}-1} r^{d-1} =A(p,d) \lim_{r\to 0}  r^{\sigma-1} r^{\frac{d}{2}-1}=0 . $$
Thus, we integrate \eqref{ODE-PsiPhi1} between $0$ and $r$ and we obtain
$$\Psi' r^{d-1} e^{r^2} = -(\lambda+d )\int_0^{r} \Psi e^{s^2} s^{d-1} ds -\int_0^{r} \Psi^{p}  e^{\frac{ps^2}{2}}s^{d-1} ds  -\int_0^{r} \Psi^{q}  e^{\frac{qs^2}{2}}s^{d-1} ds <0, \quad r>0. $$
This shows that $\Psi'<0$ and $\Psi$ is strictly monotone.  Next, we denote by $Z(r):=\frac{\Psi'(r)}{\Psi(r)}$. Assume that $Z(r)$ oscillates for $r$ large and $\liminf_{r\to \infty} Z(r)=-\alpha$, with $\alpha>0$. Note that if $\alpha=0$, then $\lim_{r\to \infty} Z(r)=0$ since $Z<0$. We will discuss later this case.
Coming back to $\alpha>0$, we infer that there exists a sequence $R_n \to \infty$ as $n\to \infty$, such that $Z'(R_n)=0$ and $Z(R_n)=-\alpha-\eta_n,$ where $\eta_n \to 0$ as $n\to \infty$. This means that
$$Z'(R_n)= \frac{\Psi''(R_n)}{\Psi(R_n)}-Z(R_n)^2=0.$$
We divide the previous equality by $Z(R_n)$ in order to get
\begin{equation}
\label{Z'-R_n}
\frac{\Psi''(R_n)}{\Psi'(R_n)}=Z(R_n)=-\alpha-\eta_n.
\end{equation}
On the other hand, from \eqref{ODE-PsiPhi} we write
\begin{align*}
\frac{\Psi''(R_n)}{\Psi'(R_n)} = & -\left(\frac{d-1}{R_n}+2R_n \right) - \frac{(\lambda+d)}{Z(R_n)} -   \exp\left(\frac{(p-1)}{2}R_n^2\right) \frac{\Psi^{p-1}(R_n)}{Z(R_n)} \\
& -   \exp\left(\frac{(q-1)}{2}R_n^2\right) \frac{\Psi^{q-1}(R_n)}{Z(R_n)}.
\end{align*}
Hence, combining this with \eqref{exp-bound-Psi}, we obtain
$$\frac{\Psi''(R_n)}{\Psi'(R_n)} \lesssim -\left(\frac{d-1}{R_n}+2R_n \right) + \frac{\lambda+d}{\alpha+\eta_n} +\frac{1}{\alpha+\eta_n}   .$$
%where $C$ is the constant in \eqref{exp-bound-Psi}.
Thus, for $n$ large enough we infer that
$$\frac{\Psi''(R_n)}{\Psi'(R_n)} < -(\alpha+\eta_n).$$
This is contradiction with \eqref{Z'-R_n} which is sufficient to conclude \eqref{limZ-infty}.
More precisely, we have
\begin{equation}
\label{limPsi'-Psi}
\lim_{r\to \infty} Z(r)=-\infty, \quad {\rm with} \quad Z(r)=\frac{\Psi'(r)}{\Psi(r)}.
\end{equation}
Indeed, we assume by contradiction that the limit is finite. From \eqref{ODE-PsiPhi}, the equation of $Z$ is given by
$$Z'+2rZ = -(\lambda + d) - f(Z)$$
where
$$f(Z):= Z^2+\frac{d-1}{r} Z +\exp\left(\frac{(p-1)}{2}r^2\right) \Psi^{p-1} +\exp\left(\frac{(q-1)}{2}r^2\right) \Psi^{q-1} .$$
Thus,
$$Z(r)=Z(1)e^{-r^2} -\int_1^r \left((\lambda + d) + f(Z(s))\right) e^{s^2} ds\  e^{-r^2} .$$
Hence, using l'Hopital's rule and \eqref{exp-bound-Psi}, we obtain
\begin{align}
\lim_{r\to \infty} Z(r)&=-  \lim_{r\to \infty}  \frac{\int_1^r \left((\lambda + d) + f(Z(s))\right) e^{s^2} ds}{ e^{r^2}} = - \lim_{r\to \infty}  \frac{f(Z(r))}{2r} \nonumber\\
&=- \lim_{r\to \infty} \left( \frac{Z^2(r)}{2r}+\frac{d-1}{2r^2}Z(r)\right). \label{limitZ}
\end{align}
Yields,
$$\lim_{r\to \infty} \left( Z(r)+\frac{Z^2(r)}{2r}+\frac{d-1}{2r^2}Z(r)\right) =0.$$
As a consequence, we get $\lim_{r\to \infty}  Z(r) =0$. Going back to the first equality in \eqref{limitZ}, similarly we have
$$\lim_{r\to \infty}  rZ(r) =-\lim_{r\to \infty}  \frac{\int_1^r \left((\lambda + d) + f(Z(s))\right) e^{s^2} ds}{ r^{-1} e^{r^2}} = -\frac{(\lambda + d) }{2}.$$
Hence, given $\varepsilon>0$,
$$Z(r)=\frac{\Psi'(r)}{\Psi(r)}\geq -\frac{(\lambda+d)+2\varepsilon}{2r}$$
for $r$ large enough. Thus, $\Psi(r) \geq r^{-\frac{(\lambda + d) }{2}-\varepsilon}$. This is a contradiction with \eqref{exp-bound-Psi}. This ends the proof of \eqref{limPsi'-Psi}. In addition, we have
%\begin{claim}
\begin{equation}
\label{limPsi'-rPsi}
\lim_{r\to \infty} \frac{\Psi'(r)}{r\Psi(r)}=-2.
\end{equation}
Indeed, using l'Hopital's rule and the equation \eqref{ODE-PsiPhi}, we get
\begin{align*}
\lim_{r\to \infty} \frac{\Psi'(r)}{r\Psi(r)}&= \lim_{r\to \infty} \frac{\Psi''(r)}{r\Psi'(r)+\Psi(r)}\\
=&\lim_{r\to \infty} \Bigg( \frac{-\left(\frac{d-1}{r}+2r \right)\Psi'(r) -(\lambda+d ) \Psi(r) +\exp\left(\frac{(p-1)}{2}r^2\right)\Psi^{p}(r)}{r\Psi'(r)+\Psi(r)}\\
& \qquad +\frac{\exp\left(\frac{(q-1)}{2}r^2\right) \Psi^{q-1}}{r\Psi'(r)+\Psi(r)} \Bigg)\\
&=\lim_{r\to \infty} \Bigg( \frac{-\left(\frac{d-1}{r^2}+2 \right) -(\lambda+d ) \frac{\Psi(r)}{\Psi'(r)} +\exp\left(\frac{(p-1)}{2}r^2\right)\Psi^{p-1}(r)\frac{\Psi(r)}{\Psi'(r)}}{1+\frac{\Psi(r)}{r\Psi'(r)}}\\
& \qquad +\frac{\exp\left(\frac{(q-1)}{2}r^2\right)\Psi^{q-1}(r)\frac{\Psi(r)}{\Psi'(r)}}{1+\frac{\Psi(r)}{r\Psi'(r)}}\Bigg)
\end{align*}
Hence, from \eqref{exp-bound-Psi} and \eqref{limPsi'-Psi}, we  conclude \eqref{limPsi'-rPsi}. Note that, combining \eqref{exp-bound-Psi} with \eqref{limPsi'-rPsi}, we deduce that
\begin{equation}
\label{exp-bound-Psi'}
|\Psi'(r)|\lesssim r e^{-\frac{r^2}{2}}.
\end{equation}
%\end{proof}
Next, we define the function
\begin{equation}
\label{def:E}
E(r):=r\Psi'(r)+2r^2\Psi(r).
\end{equation}
We claim that
\begin{equation}
\label{lim:E/psi}
\lim_{r\to \infty} \frac{E(r)}{\Psi(r)}=\frac{\lambda-d}{2}.
\end{equation}
Indeed, from \eqref{exp-bound-Psi} and \eqref{exp-bound-Psi'} we know that $\lim_{r\to \infty}E(r)=\lim_{r\to \infty}\Psi(r)=0$, thus using l'Hopital's rule, \eqref{ODE-PsiPhi}, \eqref{exp-bound-Psi} and \eqref{limPsi'-rPsi} in order to compute
\begin{align}
&\lim_{r\to \infty} \frac{E(r)}{\Psi(r)}=\lim_{r\to \infty} \frac{E'(r)}{\Psi'(r)}= \lim_{r\to \infty} \frac{\Psi'(r)+r\Psi''(r)+4r\Psi(r)+2r^2\Psi'(r)}{\Psi'(r)}\nonumber\\
&= 2-d-\lim_{r\to \infty} \frac{r\Psi(r)}{\Psi'(r)}\left((\lambda+d-4)+\exp\left(\frac{(p-1)}{2}r^2\right)\Psi^{p-1}+\exp\left(\frac{(q-1)}{2}r^2\right)\Psi^{q-1}\right)\nonumber\\
&= 2-d+\frac{(\lambda+d-4)}{2}=\frac{\lambda-d}{2}.\label{Compute-limE/u}
\end{align}
In addition, we claim that
\begin{equation}
\label{lim:r2E/psi}
\lim_{r\to \infty} r^2 \left(\frac{E(r)}{\Psi(r)}-\frac{\lambda-d}{2}\right)=\frac{\lambda-d}{8}(\lambda+d-4).
\end{equation}
Indeed, following the same computations as in \eqref{Compute-limE/u}, we have
\begin{align*}
\lim_{r\to \infty} r^2\left(\frac{E(r)}{\Psi(r)}-\frac{\lambda-d}{2}\right)&= (4-\lambda-d)\lim_{r\to \infty} r^2\left(\frac{r\Psi(r)}{\Psi'(r)}+\frac{1}{2}\right)\\
&=\left(2-\frac{\lambda+d}{2}\right)\lim_{r\to \infty}r^2 \left(2+\frac{\Psi'(r)}{r\Psi(r)}\right)\frac{r\Psi(r)}{\Psi'(r)}\\
&=\left(2-\frac{\lambda+d}{2}\right)\lim_{r\to \infty}\frac{E(r)}{\Psi(r)}\lim_{r\to \infty}\frac{r\Psi(r)}{\Psi'(r)}\\
&=\frac{\lambda-d}{8}(\lambda+d-4).
\end{align*}
Finally, combining \eqref{lim:E/psi} and \eqref{lim:r2E/psi}, we obtain
$$\frac{E(r)}{\Psi(r)}= \frac{\lambda-d}{2} +\frac{\lambda-d}{8}(\lambda+d-4) \frac{1}{r^2} + o\left(\frac{1}{r^2}\right) \quad {\rm as} \ r\to \infty. $$
Yields, from \eqref{def:E}, that
$$\frac{\Psi'(r)}{\Psi(r)}=-2r+\frac{\lambda-d}{2r} +\frac{\lambda-d}{8}(\lambda+d-4) \frac{1}{r^3} + o\left(\frac{1}{r^3}\right) \quad {\rm as} \ r\to \infty. $$
Hence, integrating between $1$ and $r$, we obtain
\begin{equation}
\label{behv:infty:Psi}
\Psi(r)= K e^{-r^2} r^{\frac{\lambda-d}{2}} \left(1+O\left(\frac{1}{r^2}\right)\right), \textrm{ as } r\to+\infty,
\end{equation}
with $K=\Psi(1)\exp\left\{1+ \frac{\lambda-d}{16}(\lambda+d-4)+\int_1^{\infty} o\left(s^{-3}\right) ds\right\}.$ Thus we finish the proof of \eqref{behv:infty:Phi} by combing \eqref{change1-Phi} with \eqref{behv:infty:Psi}.
\end{proof}

\noindent
Yakine Bahri
\\
Department of Mathematics and Statistics
\\ 
University of Victoria 
\\
3800 Finnerty Road, Victoria, B.C., Canada V8P 5C2 
\\
E-mail: ybahri@uvic.ca

\vspace{0.5cm}

\noindent
Hichem Hajaiej
\\
Department of Mathematics 
\\ 
California State University, Los Angeles, USA
\\
5151 State Drive, 5151 Los Angeles, 90331 California, USA
\\
E-mail:hhajaie@calstatela.edu


\begin{thebibliography}{10}

\bibitem{ABSA}
F. K. Abullaev and M. Salerno. Gap-Townes solitons and localized excitations in low dimensional Bose-Einstein condensates in optical lattices. Phys. Rev. A, 72:033617, 2005 

\bibitem{BMR}
Bahri, Y., Martel, Y. and Raphaël, P. Self-similar Blow-Up Profiles for Slightly Supercritical Nonlinear Schr\"{o}dinger Equations. Ann. Henri Poincaré (2021).

\bibitem{BFPS}
P. Bizon, F. Ficek, D.E. Pelinovsky and  S. Sobieszek. Ground state in the energy super-critical Gross-Pitaevskii equation with a harmonic potential. (2020) arXiv preprint arXiv:2009.04929.

\bibitem{BPT}
H. Brezis, L. A. Peletier, and D. Terman. A very singular solution of the heat equation with absorption. D. Arch. Rational Mech. Anal. (1986) 95: 185.

\bibitem{budnor} Budd, C.; Norbury, J., Semilinear elliptic equations and supercritical growth, J. Differential
Equations 68 (1987), no. 2, 169--197.

\bibitem{CRS} C. Collot, P. Raphaël and J. Szeftel. On the stability of type I blow up for the energy super critical heat equation. Mem. Amer. Math. Soc. 260 (2019), no. 1255.

\bibitem{DGW}
E. N. Dancer, Z. Guo, and J. Wei. Non-Radial Singular Solutions of the Lane-Emden Equation in ℝ N. Indiana University Mathematics Journal 61, no. 5 (2012): 1971-996.
    
\bibitem{Di} Ding, W. Y., Ni, W. M. (1985). On the elliptic equation $\Delta u+ Ku^{(n+ 2)/(n-2)}= 0$ and related topics. Duke mathematical journal, 52(2), 485-506.
    
\bibitem{Foued2}
F. Hadj Selem. Radial solutions with prescribed numbers of zeros for the nonlinear Schr\"{o}dinger equation with harmonic potential. Nonlinearity. 2011, Vol.24, No.6, p.1795.

\bibitem{HSKW}
F. Hadj Selem, H. Kikuchi, and J. Wei. Existence and uniqueness of singular solution to stationary Schr\"{o}dinger equation with supercritical non-linearity. Discrete \& Continuous Dynamical Systems - A,33,10,4613,4626,2013-4-1, vol. 2, p. 1.

\bibitem{HHMT}
F. Hadj selem, H. Hajaiej, P. Markowich, S. Trabelsi Variational Approach to the orbital stability of standing waves of Gross-Pitaevskii equation.  Milan Journal of Mathematics 82, 273–295(2014)pages273–295(2014).

\bibitem{HiroseOhta}
M. Hirose and M. Ohta
Uniqueness of positive solutions to scalar field equations with harmonic potential
Funkcial. Ekvac., 50 (2007), pp. 67-100

\bibitem{J}
L. Jeanjean and T. Trung Le, Multiple normalized solutions for a Sobolev critical Schr\"{o}dinger equation,  arXiv:2011.02945v1.

\bibitem{Kikuchi}
H. Kikuchi, Existence of standing waves for the nonlinear Schr\"{o}dinger equation with double power nonlinearity and harmonic potential,  Advanced Studies in Pure Mathematics, 2007: 623-633 (2007).

\bibitem{Li-Wei}
Y. Li and W. M. Ni, Radial symmetry of positive solutions of nonlinear elliptic equations in $\R^n$, Communications in Partial Differential Equations, 18:5-6, 1043-1054  (1993).

\bibitem{L}
Mathieu Lewin and Simona Rota Nodari. The double-power nonlinear Schrodinger equation and its generalizations: uniqueness, ¨ non-degeneracy and applications. arXiv:2006.02809, 2020

\bibitem{Jo} 
Joseph, D. D., Lundgren, T. S. (1973). Quasilinear Dirichlet problems driven by positive sources. Archive for Rational Mechanics and Analysis, 49(4), 241-269.

\bibitem{MAL}  
B. Malomed. Vortex solitons: Old results and new perspectives. Physica D, 399(1):108–137, 2019
    
\bibitem{MePe}  
F. Merle and L. A. Peletier, Positive solutions of elliptic equations involving supercritical growth, Proceeding of the Royal Society of Edingburgh, 118A (1991), 49-62.
    
\bibitem{S1} 
Nicola Soave. Normalized ground states for the NLS equation with combined nonlinearities. J. Differential Equations, 269(9):6941– 6987, 2020.

\bibitem{S2}
 Nicola Soave. Normalized ground states for the NLS equation with combined nonlinearities: the Sobolev critical case. J. Funct. Anal., 279(6):108610, 43, 2020.
 
 \bibitem{S}
 Atanas Stefanov. On the normalized ground states of second order PDE’s with mixed power non-linearities. Comm. Math. Phys., 369(3):929–971, 2019
 
\bibitem{YiLi} Y. Li. Asymptotic Behavior of Positive Solutions of Equation $\Delta u +K(x)u^p=0$ in $\mathbb{R}^n$. J. Differential Equations, 95(2):304-330, 1992.
\end{thebibliography}
\end{document}